\documentclass[12pt]{amsart}
\usepackage[utf8]{inputenc}
\usepackage{graphicx}
\usepackage{amsfonts}
\usepackage{amsmath}
\usepackage{amsthm}
\usepackage{amssymb}
\usepackage{hyperref}
\usepackage{mathtools}
\usepackage{geometry}
\usepackage{comment}
\usepackage{subfigure}
\usepackage{pgfplots}

\usepackage{booktabs}
\usepackage{xtab}
\usepackage{xcolor}

\theoremstyle{definition}
\newtheorem{theorem}{Theorem}[section]
\newtheorem{corollary}[theorem]{Corollary}
\newtheorem{lemma}[theorem]{Lemma}
\newtheorem{proposition}[theorem]{Proposition}
\newtheorem{definition}[theorem]{Definition}
\newtheorem{remark}[theorem]{Remark}
\newtheorem{problem}[theorem]{Problem}
\newtheorem{assumption}[theorem]{Assumption}

\DeclareMathOperator{\Tr}{Tr}
\newcommand{\Div}{\mbox{\rm div}}

\newcommand{\R}{\mathbb{R}}
\newcommand{\go}{{\Gamma_0}}

\newcommand{\id}[1]{{\rm id}_{#1}}
\newcommand{\D}[1]{\underline{D}_{#1}}

\newcommand{\divg}{{\rm div}_\Gamma}
\newcommand{\goh}{{\Gamma_0^h}}
\newcommand{\Divg}{{\rm div}_\go}

\newcommand{\dee}{~{\rm d}}
\newcommand{\ud}{{\rm d}}
\newcommand{\XX}{\mathcal{X}}
\newcommand{\conormal}{{\rm n}}

\newcommand{\jump}[1]{{[\![{#1}]\!]}}
\newcommand{\avg}[1]{\{\!\{{#1}\}\!\}}

\newcommand{\orcidlogo}{\includegraphics[height=10pt]{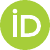}}
\newcommand{\orcid}[1]{\href{https://orcid.org/#1}{\orcidlogo}}

\date{\today}
\title[Small deformations of a Canham-Helfrich tube]{Small deformations of a near cylindrical tube for the Canham-Helfrich Energy with applications to biological membranes}

\author{Charles M.~Elliott}
\address{
  \orcid{0000-0002-6924-4455}
  Mathematics Institute, Zeeman Building, University of Warwick, Coventry. CV4 7AL, United Kingdom}
\thanks{The work of CME was partially supported by the Royal Society via a Wolfson Research Merit Award.}

\author{Carsten Gr\"aser}
\address{
  \orcid{0000-0003-4855-8655}
  Department Mathematik, Friedrich-Alexander-Universit\"at Erlangen-N\"urnberg, Cauerstra{\ss}e 11, 91058 Erlangen, Germany}
\thanks{This research has been
    supported 
    by Deutsche Forschungsgemeinschaft (DFG)
    through the grant CRC 1114: “Scaling Cascades in Complex Systems”,
    Project Number 235221301,
    Project A07 "Langevin dynamics of particles in membranes".}

\author{Philip J.~Herbert}
\address{
  \orcid{0000-0002-6513-1728}
  Department of Mathematics, University of Sussex, Brighton, BN1 9RF, United Kingdom}\email{p.herbert@sussex.ac.uk}


\begin{document}
\begin{abstract}
In this article we develop a quadratic energy which approximates the Canham-Helfrich energy for a tube-like surface with clamped boundary and area constraint.
The energy is suited to the study of small deformations of biological membranes where the deformations are induced by point forces or point constraints due to the cytoskeleton or a phase dependent spontaneous curvature.
Since the deformations we consider are small, we may assume that the surface of interest is a graph over a fixed, undeformed surface.
A Lagrangian and the associated Euler-Lagrange equations for the graph are derived.
Well-posedness of the Euler-Lagrange equations in suitable spaces is shown.
Finally, we provide some illustrative numerical examples.
\end{abstract}
\maketitle

\section{Introduction}
In this paper, we study small deformations of tubular surfaces, $\Gamma$, described by Helfrich-type energies.
It is assumed that surface deformations are induced by small external forces leading  to  a deformed surface that can  be described by a graph over a reference  tube.
Here we consider a surface  with boundary, where we assume that $\partial \Gamma$ is given together with the conormal. 
This work will derive a quadratic approximation to the Helfrich-type energy in terms of the height function of the graph over this tube, leading to a linear fourth order  PDE on the reference surface in contrast to a non-linear geometric PDE.
We will show that this linear PDE is well-posed and consider different forcing, including 
point forcing, point constraints and forcing along a curve.
A non-linear phase field perturbation is also considered.

This work is motivated by the setting where  the surface models  a lipid-bilayer  forming   a membrane tubule or part of the boundary  of a cell. We suppose  that the membrane is deformed by small forces from a reference tube that is a local minimiser of the energy.
These forces may be located at points or along curves and the force may be due to proteins attached to the surface.
Another forcing we will consider is  due to the effect of saturated lipids leading to a surface gradient energy.
Since the morphology of the cell membrane is well-known to be coupled to embedded or attached  proteins, the modelling of small deformations of the membrane may help to give qualitative descriptions of the interplay.
We highlight \cite{McmGal05} which discusses the main mechanisms of membrane deformation.

We are interested in the case where the surface  has a boundary.
This is motivated by a setting wherein the membrane may be a medium transporting proteins between two larger vesicles.
These structures are known as membrane tubes (or membrane tubules) and  are common structures in many cellular organelles with diverse roles within a cell.
We consider  surfaces with fixed  area that are  clamped at the boundary  where  they may attach to  larger vesicles.
In this setting the area constraint may be justified by considering that the tube is at equilibrium.

It is known that the shape of the lipid bilayer which surrounds a cell is well-modelled by the Canham-Helfrich energy \cite{Can70,Hel73}
\begin{equation}\label{eq:CH}
	E(\Gamma) := \int_\Gamma \left( \frac{\kappa}{2} (H-H_s)^2 + \kappa_G K \right) \!\!\dee \Gamma,
\end{equation}
where the membrane is assumed to be sufficiently thin that it may be modelled by a two-dimensional hypersurface $\Gamma\subset \R^3$ with mean and Gaussian curvatures $H$ and $K$ respectively.
We take $H$ to be the sum of the principle curvatures, twice the typical value.
The constants $\kappa>0,\,\kappa_G\in \R$ are the bending rigidities respectively due to the mean and Gaussian curvatures.
The value $H_s \in \R$ is the spontaneous curvature and corresponds to an asymmetry in the lipid bilayer, for example due to a different concentration of saturated and unsaturated lipids.
In the above, we consider that $H_s$ will be a given small function.
This smallness assumption appears in the works \cite{EllHat21,EllHatHer20,CaeEllGra26} which all consider models with small deformations around the sphere.

\subsection*{Main contribution}
The main contribution of this work is to produce a linear fourth order PDE which describes small deformations of a membrane tube.
Associated to this, we show the well-posedness of the PDE, and include a variety of biologically relevant forcing.
Implicitly this shows that the Helfrich-Cylinder is a stable local minimiser of the Willmore energy with a specific surface tension, which has been shown in an axi-symmetric regime.

\subsection*{Membrane tubes}
Tubular structures have been identified as early as the 1950s under examination by electron microscopy \cite{Por53}, where they are seen to be part of the endoplasmic reticulum, a network of interconnected tubules and cisternae.
Membrane tubes can be generated by in-vitro tether-pulling experiments \cite{CulChiBas05,LiAnvTak02,SunGraHeg05,DerJulPro02}.
It has been demonstrated that cells exchange enclosed material via the formation of similar narrow fluid membrane tubes known as tunnelling nanotubes \cite{MarGouZur12,GerCar08}.
These membrane tubes typically have a diameter of 50 nm to 200 nm and may extend over tens of microns.
Tubular membranes have also been studied in static and dynamical settings.
In particular, linearised descriptions on cylindrical geometries have been used to analyse the response of membrane tubes to localised forces, active processes, and curvature-inducing inclusions, highlighting the role of geometry in determining deformation modes and stability \cite{AlIRowSen18, AlISenTur20, JanLieAlI24}.
The article \cite{Wei18} discusses how proteins on the membrane are important for tubulation - the formation of tubes - to occur on a membrane and \cite{BahHum17} explores the formation and stability of tubes.
For further details on the physics of membrane tubes, we refer the reader to \cite{Rou13}.

\subsection*{Related work}
In \cite{EllFriHob17}, the authors derive a fourth order PDE which qualitatively describes the shape of a near-spherical or near-toroidal membrane under the assumption of small deformations.
In \cite{Kie19}, the author develops a quadratic energy for a tube based on the Taylor expansion ($f(1)\approx f(0) + f'(0)+ \frac{1}{2}f''(0)$) of the Canham-Helfrich Energy \eqref{eq:CH} for arbitrary $H_s$.
The work \cite{EllGraHob16} neglects Gauss curvature and boundary contributions to consider a geometric linearisation of \eqref{eq:CH}, the so-called Monge gauge, and shows well-posedness with point particle constraints and curve constraint models amongst other models.
 
The quadratic models mentioned above have been used to consider membrane mediated interactions of objects attached to the membrane.
The works of \cite{GraKie17,EllHer21} consider shape derivatives of simplified membrane energies, the geometric linearisation of \cite{EllGraHob16} or small deformation energy of \cite{EllFriHob17}, where the shape is respectively the location of curve and point constraints as in \cite{EllGraHob16} and \cite{EllHatHer20}.

The work of \cite{Sch10} considers the minimisation of a Willmore energy when the boundary is clamped,
existence of a smooth embedded minimiser is shown for sufficiently small energy.
The minimisation of the Willmore functional for graphs with either Navier
or clamped boundary data is considered in \cite{DecGruRog18}.

Axi-symmetry is exploited to gain additional structure.
The article \cite{DecDoeGru21} studies the structure of minimisers of the Helfrich energy with surface tension and clamped boundary in an axi-symmetric setting, in particular solutions near a cylinder c.f.~Section 4. See also \cite{MeiUec25}.
The work of \cite{BarGarNur19} deals with the gradient flow of the Willmore functional for both open and closed surfaces in an axi-symmetric setting.
In \cite{DecGru09a} the authors are interested in the Willmore functional with Navier boundary.
They are also interested in small deformations of this energy, considering the second variation of their axi-symmetric energy at a section of a Clifford torus.
The axi-symmetry leads to a linear fourth order ODE for which the general solution is given.
In \cite{Ols23}, the author combines Helfrich and tangential surface Navier--Stokes models to obtain a system describing the surface under constraints, for which axi-symmetric shapes are admissible.

For historical details and modelling aspects of membranes with boundary, we refer to \cite{Nit93}.

\subsection*{Outline}
We start in Section \ref{sec:Notation} by introducing notation and definitions for surface PDE along with a few necessary results.
In Section \ref{sec:Modelling}, we give the derivation for a quadratic Lagrangian, following this in Section \ref{sec:critPoints} is a discussion of the existence of critical points of the derived Lagrangian. 
We then, in Section \ref{sec:ExampleForcing} provide a series of examples of forcings which may give the membrane a small deformation.
We conclude with Section \ref{sec:numExp} with numerical experiments demonstrating the examples of the preceding section.

\section{Notation and preliminaries}\label{sec:Notation}
Let $\Gamma$ be a two-dimensional connected, orientable hypersurface in $\R^3$, in particular $\partial\Gamma$ need not be empty.
We assume that $\Gamma$ and its boundary are as regular as required, but at most $C^4$.
The unit normal to $\Gamma$ is denoted by $\nu$, where we make an arbitrary choice of direction.
We define $P_\Gamma := I - \nu \nu^T$  to be, at each point of $\Gamma$, the projection onto the tangent space at that point, where $I$ is the identity matrix on $\R^3$.

For a differentiable function $f \colon \Gamma \to \R$, we define the tangential gradient of $f$ by
\[
	\nabla_\Gamma f := P_\Gamma \nabla \tilde{f}|_{\Gamma},
\]
where $\tilde{f}$ is a differentiable extension of $f$ to an open neighbourhood of $\Gamma$.
Here $\nabla$ is the standard gradient in $\R^3$.
It holds that the above definition does not depend on the extension $\tilde{f}$ \cite[Lemma 2.4]{DziEll13}.
The components of the tangential gradient are denoted $(\D{1}f,\D{2}f,\D{3} f)^T := \nabla_\Gamma f$.
For a continuously differentiable function $\eta \colon \Gamma \to \R^3$, the tangential divergence $\divg$ is defined by
\[
    \divg \eta = \sum_{i=1}^3 \D{i} \eta_i.
\]
For a twice differentiable function the Laplace-Beltrami operator is defined
\[
	\Delta_\Gamma f := \divg \nabla_\Gamma f
\]
and the surface Hessian is defined
\[
    (D^2_\Gamma f)_{ij} := \D{i}\D{j} f \mbox{ for } i,j = 1,2,3,
\]
which need not be symmetric c.f.~\cite[Lemma 2.6]{DziEll13}.

It is known \cite[Lemma 2.8]{DziEll13} that there is a neighbourhood, $\mathcal{N}_\delta$, of $\Gamma$, with width $\delta>0$ and maps $d\colon \mathcal{N}_\delta \to \R$, $\pi \colon \mathcal{N}_\delta \to \Gamma$ such that for any $x \in \mathcal{N}_\delta$ the decomposition
\begin{equation}\label{eq:orientedDistanceFunction}
	x = \pi(x) + d(x) \nu(\pi(x))
\end{equation}
is unique.
The function $d$ is known as the oriented distance function and $\pi$ the closest point projection.

The map $\mathcal{H}:= \nabla_\Gamma \nu$ is called the extended Weingarten map, which is symmetric and $\mathcal{H} \nu =0$.
The eigenvalues of $\mathcal{H}$, which belong to the tangential eigenvectors, $\kappa_1,\kappa_2$, are the principle curvatures of $\Gamma$.
The mean curvature $H = \Tr(\mathcal{H})$ is the sum of the principle curvatures, and the Gauss curvature $K = \det(\mathcal H)$, the product.
We denote the identity map on $\Gamma$ by $\id{\Gamma}$ and recall that $-\Delta_\Gamma \id{\Gamma} = H \nu$.

We say $f \in L^1(\Gamma)$ has a weak derivative $\D{i}f:= v_i\in L^1(\Gamma)$ if
\[
	\int_\Gamma f \D{i} \phi \dee \Gamma = \int_\Gamma \left( - v_i \phi + f \phi H \nu_i \right) \!\!\dee \Gamma
\]
for all smooth functions $\phi\colon \Gamma \to \R$ with compact support in $\Gamma$.
The integrals are taken with respect to the two-dimensional Hausdorff measure on $\Gamma$.
We also have the Stokes formula,
\begin{equation}\label{eq:StokesFormula}
	\int_\Gamma \nabla_\Gamma f \dee \Gamma = \int_\Gamma f H \nu \dee \Gamma + \int_{\partial \Gamma} f \conormal \dee \mathcal{H}^1,
\end{equation}
where $\conormal$ is the outward unit co-normal which is normal to $\partial\Gamma$ and tangent to $\Gamma$.
The Sobolev spaces $H^1(\Gamma)$ and $H^2(\Gamma)$ on the hypersurface $\Gamma$ are defined by
\begin{align*}
	H^1(\Gamma) :=& \{ f \in L^2(\Gamma) : \D{i} f \in L^2(\Gamma),\, i=1,2,3\},
	\\
	H^2(\Gamma) :=& \{ f \in H^1(\Gamma) : \D{i}\D{j} f \in L^2(\Gamma),\, i,j=1,2,3\}.
\end{align*}
These spaces are Hilbert spaces when endowed with appropriate inner products and norms.
We choose to use the following inner products and norms,
\begin{align*}
	(u,v)_{1,2}
	:=&
	\int_\Gamma \left( \nabla_\Gamma u \cdot \nabla_\Gamma v + u v \right)\!\! \dee \Gamma
	~\mbox{and}~
	\|u\|_{1,2}:= \sqrt{(u,u)_{1,2}},
	\\
	(u,v)_{2,2}
	:=&
	\int_\Gamma \left( \Delta_\Gamma u \Delta_\Gamma v +  \nabla_\Gamma u \cdot \nabla_\Gamma v + u v \right)\!\! \dee \Gamma
	~\mbox{and}~
	\|u\|_{2,2}:= \sqrt{(u,u)_{2,2}}.
\end{align*}
Notice that this is not the standard inner product on $H^2(\Gamma)$ which contains mixed second order derivatives.
On a closed surface this given norm is equivalent to the standard norm \cite[Lemma 3.2]{DziEll13}.
The equivalence also holds on a surface with boundary if one has appropriate boundary conditions, for example vanishing trace, which follows from Lemma \ref{lem:EquivNormDirichlet}.
\begin{corollary}\label{cor:equivNorms}
    There is $C>0$ such that
	\[
		\|v\|_{H^2(\Gamma)}\leq C\|v\|_{2,2} \quad \forall v \in H^2_0(\Gamma).
	\]
	In particular the norms $\|\cdot\|_{2,2}$ and $\|\cdot \|_{H^2}$ are equivalent on $H^2_0(\Gamma)$.
\end{corollary}
\begin{definition}\label{def:GeoQuant}
We now define the following geometric quantities
\begin{align*}
	W(\Gamma) := \frac{1}{2}\int_{\Gamma}H^2 \dee \Gamma,\quad
	A(\Gamma) := \int_{\Gamma} 1\dee \Gamma,\quad
	G(\Gamma) := \int_\Gamma K\dee \Gamma,
\end{align*}
the Willmore energy, surface area, and the integral of the Gauss curvature of $\Gamma$, respectively.
We then define, for $\kappa>0$, $\kappa_G \in \R$,
\begin{align*}
	\mathcal{E}(\Gamma) := \kappa W(\Gamma) + \kappa_G{G}(\Gamma).
\end{align*}
\end{definition}
It is noted that $\kappa_G \in [-\kappa,0]$ is the physically relevant range \cite{Nit93}.
The case $\kappa_G = 0$ corresponds to the standard Willmore energy, and $\kappa_G = -\kappa$ corresponds to the square integral of so-called deviatoric curvature (the difference of principle curvatures) \cite{Fis92}.
It is however noted that by the Gauss-Bonnet theorem, it holds that $G(\Gamma)$ is constant for surfaces with a clamped boundary as we will consider in this work, therefore we set $\kappa_G = 0$.

\section{Modelling of small surface deformations}\label{sec:Modelling}
We pose the problem:-\\
Given a small force $\rho \mathcal{F}$ and a small {\it spontaneous curvature} $\rho \tilde H_0 \colon \R^3 \to \R$, we wish to find a surface $\Gamma$ which minimises
\begin{equation}\label{eq:rhoForcing}
	\mathcal{J}_\rho(\Gamma):= \mathcal{E}(\Gamma) - \kappa \int_\Gamma \rho \tilde H_0 H \dee \Gamma + \frac{\kappa}{2}\int_\Gamma \rho^2 \tilde H_0^2 \dee \Gamma - \rho \mathcal{F}(\Gamma),
\end{equation}
over surfaces $\Gamma$ which have a prescribed surface area and whose boundaries are clamped.
In order to remove the area constraint, let us consider consider the Lagrangian
\begin{equation}
    {L}_\rho(\Gamma, \sigma)
    =
    \mathcal{J}_\rho(\Gamma)
    +
    \sigma (A(\Gamma) - A_*),
\end{equation}
for some $A_*>0$, the prescribed area, and $\sigma \in \R$ a Lagrange multiplier.
Finding critical points of ${L}_\rho$ is challenging.
It is natural to assume that a sufficiently small force applied to the membrane will result in a perturbed membrane which may be described as a graph over $\go$, a critical surface of $L_0$.

\subsection{Ansatz Surface}

Given $R,L>0$ let us denote the Helfrich-Cylinder by
\begin{equation}\label{eq:TheTube}
	\go := \left\{ x \in \R^3 : x_1^2 + x_2^2 = R^2,\, x_3 \in \left(-\frac{L}{2},\frac{L}{2}\right)\right\}.
\end{equation}
As $\go$ is a surface with boundary, take $\nu \colon \go \to \R^3$ to be given by $\nu(x) = \frac{1}{R}(x_1,x_2,0)^T$.
We note that $\go$ does not include its boundary.
The choice that the tube goes from $-\frac{L}{2}$ to $\frac{L}{2}$ is arbitrary, as is the choice that the length is along the $x_3$ axis.
Let us note that $\go \cong \mathbb{S}^1\times (0,L)$ and it may be convenient to use cylindrical polar coordinates,
\begin{equation}\label{eq:convenientParameterisation}
	X\colon	[0,2\pi) \times \left(- \frac{L}{2},\frac{L}{2}\right) \to \go, \quad
	X\colon	(\theta, z) \mapsto \left(R\sin(\theta), R \cos(\theta),z\right),
\end{equation}
where we understand $X$ to be $2\pi$-periodic in its first variable.
For $g\in C^1(\go)$ and $G:= g \circ X$, the parametrisation of $g$ by $X$, we have
\[
	(\nabla_\go g) \circ X= \frac{(X_2,-X_1,0)^T}{R^2} \partial_\theta G +  (0,0,1)^T \partial_z G.
\]
It is convenient to write the unit azimuthal vector
\begin{equation}\label{eq:tauDefinition}
	\tau \colon \go \to \R^3, \quad \tau(x):= \frac{(x_2,-x_1,0)^T}{R}
\end{equation}
and denote $\hat{z}(x) = (0,0,1)^T$, which together form a basis of the tangent space of $T_x \go$ for all $x \in \go$.

For convenience, the oriented distance function for the tube can be given explicitly by $d(x) := \sqrt{x_1^2 + x_2^2} -R$ for $x_3 \in (-\frac{L}{2},\frac{L}{2})$ which in the neighbourhood $\mathcal{U} := \{ x \in \R^3 : x_1^2 + x_2^2 >0,\, |x_3| \leq \frac{L}{2}\}$ is smooth.

Let us calculate the variation of
\begin{equation}
    L_0(\Gamma,\sigma)
    =
    \frac{\kappa}{2} W(\Gamma)
    +
    \sigma (A(\Gamma)-A_0).
\end{equation}
By making use of Proposition \ref{prop:firstVar}, it holds that
\begin{equation}\begin{split}
    \partial_\Gamma L_0(\Gamma,\sigma)[U\nu]
    =&\
    \kappa \int_\Gamma H\left( - \Delta_\Gamma U + \frac{1}{2} H^2U - |\mathcal{H}|^2 U\right) \dee \Gamma
    +
    \sigma \int_\Gamma H U  \dee \Gamma,
    \\
    \partial_\sigma L_0(\Gamma,\sigma)[\mu]
    =&\
   (A(\Gamma) - A_*)\mu
\end{split}\end{equation}
where $\partial_\Gamma$ is written to mean the derivative with respect to shape, and we evaluate in direction $U\nu$ and $\mu \in \R$ .

One may calculate that on $\go$, $H = |\mathcal{H}| = \frac{1}{R}$.
Hence, for $A_* = 2\pi R L$, we see that for variations $U$ with $\nabla_{\Gamma_0}U\cdot \conormal
=0$ on $\partial \go$
\begin{equation}
    L_0'(\go, \frac{\kappa}{2 R^2}) = 0,
\end{equation}
so the pair $(\go, \frac{\kappa}{2 R^2})$ is a critical point of $L_0$.
We will denote $\sigma_0 = \frac{\kappa}{2 R^2}$.

Later, c.f.~Theorem \ref{thm:EnergyCoercive}, we will demonstrate that $\partial_\Gamma^2L_0(\go,\sigma_0)$ is positive definite so $\go$ is not only a critical point with constrained area, but a local minimiser.
As such, it becomes reasonable to consider small perturbations of the energy.

\subsection{Small perturbations}
Following the notation of \cite{EllFriHob17}, given $v \colon \go \to \R$ sufficiently smooth and $\rho>0$ sufficiently small, we denote the graph
\begin{equation}\label{eq:DefiningGraphSurface}
	\Gamma_\rho(v):= \{ x + \rho v(x)\nu_\go(x) : x \in \go\}.
\end{equation}
We will fix $\rho \ll 1$, and look for a critical point of the form $(\Gamma_\rho, \sigma_\rho)$, where we have, for simplicity, dropped $v$ in the notation of $\Gamma_\rho(v)$. 
As $\Gamma_\rho$ is a first order perturbation of $\go$, let us consider that $\sigma_\rho = \sigma_0 + \rho \mu$.

With an ansatz surface and Lagrange multiplier $(\go, \sigma_0)$, for fixed $v$ and $\mu$ we will Taylor expand the highly non-linear Lagrangian
\begin{equation}\label{eq:LagrangianGeometric}
	\mathcal{L}_\rho(v,\mu):= \mathcal{J}_\rho(\Gamma_\rho) + (\sigma_0 + \rho \mu )(A(\Gamma_\rho) -A_*)
\end{equation}
in powers of $\rho$.
This leads to
\begin{equation}\label{eq:TaylorExpansion}
    \nonumber
    \mathcal{L}_\rho(v,\mu)
	=
	\mathcal{L}_0(v,\mu)
	+
	\rho \left.\frac{\ud \mathcal{L}_s(v,\mu)}{\ud s}\right|_{s = 0}
	+
	\frac{\rho^2}{2} \left.\frac{\ud ^2\mathcal{L}_s(v,\mu)}{\ud s^2}\right|_{s = 0}
	+
	O(\rho^3).
\end{equation}
We will see later, for reasonable assumptions on the forcing $\mathcal{F}$, that both the zero and first order terms will be independent of the pair $(v,\mu)$.
Hence for small $\rho$, higher order terms are negligible and the second order terms are the leading terms where variations of $(v,\mu)$ will be non-trivial.
Our approach will henceforth be to find a critical point of the second order terms.

In the above we have a Lagrange multiplier to enforce the area constraint, to handle the boundary conditions for the surface, we restrict the perturbation $v$ to be a function which lies in functions with vanishing trace and whose normal derivative vanishes at the boundary.

For the moment, we do not consider a specific type of force and instead provide some assumptions on the forces.
Examples of possible forces which appear in the literature are point forces \cite{EllFriHob17}, point constraints \cite{EllHatHer20,EllHer20,EllHer21} or a phase dependent spontaneous curvature \cite{EllHat21}.
These examples are discussed in Section \ref{sec:ExampleForcing}.
The following assumptions on $\mathcal{F}$ and $H_s$ are made.
\begin{assumption}\label{ass:FGo0}
Assume that:
\begin{itemize}
\item	The surface energy $\mathcal{F}$ satisfies $\mathcal{F}(\go) = 0$.
\item   There is $f \in \left( H^2(\go) \cap H_0^1(\go) \right)^*$ such that $\mathcal{F}'(\go)[v\nu] = \langle f, v\rangle$ for all $v \in H^2(\go)\cap H_0^1(\go)$, where $\langle\cdot,\cdot\rangle$ denotes the dual pairing.
\item	The function $\tilde H_0$ is constant in the normal direction around $\go$, that is to say, $\tilde H_0(x) = \tilde H_0(\pi(x))$ for all $x \in \mathcal{N}_\delta$, i.e., there is some function $c_0\colon \go \to \R$ such that $\tilde H_0 = c_0 \circ \pi $.
\end{itemize}
\end{assumption}
By Assumption \ref{ass:FGo0}, we see that
\[
    \int_{\Gamma_\rho} H \tilde H_0 \dee \Gamma = \int_{\Gamma_\rho} H c_0 \circ \pi \dee \Gamma
    \quad\mbox{and}\quad
    \int_{\Gamma_\rho} \tilde H_0^2 \dee \Gamma = \int_{\Gamma_\rho } c_0^2 \circ \pi \dee \Gamma.
\]




\subsection{The first variation}
\begin{lemma}\label{lem:FirstDerivative}
	The derivative of $\mathcal{L}_\rho$ is given by
	\begin{align*}
		\left. \frac{\ud \mathcal{L}_s(v,\mu)}{\ud s}\right|_{s=0}
		=&
        - \frac{\kappa}{R} \int_{\go} (\Delta_\go v + c_0)\dee \go
        .
	\end{align*}
\end{lemma}
\begin{proof}
	We calculate
	\begin{align*}
	\left.\frac{\ud \mathcal{L}_s(v,\mu)}{\ud  s} \right|_{s=0}
	=&
	\mathcal{E}'(\go)[v\nu] + \sigma_0 A'(\go)[v\nu]
	-
	\kappa\left. \frac{\ud}{\ud s}\right|_{s=0} \left( s \int_{\Gamma(\rho v)} H c_0 \circ \pi \dee \Gamma \right)
	\\
	&+
    \frac{\kappa}{2}\left. \frac{\ud}{\ud s}\right|_{s=0} \left( s^2 \int_{\Gamma_s} c_0^2 \circ \pi \dee \Gamma \right)
	+ \mu (A(\go) -A_*)- \mathcal{F}(\go).
	\end{align*}
	The first part of Assumption \ref{ass:FGo0} gives $\mathcal{F}(\go) = 0$,
    the assumption that $A_* = 2\pi R L = A(\go)$ means that the term involving $\mu$ vanishes.
	The first variations of $W$ and $A$ are given in Proposition \ref{prop:firstVar} and
	for the spontaneous curvature terms, one calculates,
	\[
	    \left.\frac{\ud}{\ud s}\right|_{s=0}  \left( -\kappa s \int_{\Gamma_s} H c_0 \circ \pi \dee \Gamma + \frac{\kappa s^2}{2} \int_{\Gamma_s} c_0^2 \circ \pi \dee \Gamma \right)
	    =
	    -\kappa \int_\go H c_0 \dee \go.
	\]
\end{proof}

\begin{remark}\label{rem:ObservationsOnFirstVariation}
	From \eqref{eq:StokesFormula}, it is seen $\int_\go \Delta_\go v\dee \go = \int_{\partial \go} \nabla_\go v \cdot \conormal \dee \go $.
	It follows that for any $(v,\mu) \in H^2_0(\go) \times \R$
    \begin{equation}
        \left. \frac{\ud \mathcal{L}_s(v,\mu)}{\ud s}\right|_{s=0}
        =
        -\frac{\kappa}{R}\int_\go c_0 \dee \go.
    \end{equation}
    In this setting, we see that $\int_\go c_0 \dee \go$ is a given constant, which may correspond to the saturation of the lipids within the surface.
\end{remark}

\subsection{The second variation}
We now calculate the second order contributions of $\mathcal{L}_\rho(v,\mu)$.
It is worth noting that the second variation of geometric objects is quite delicate and we refer the reader to \cite[Remark 3.2]{EllFriHob17} which highlights a particular challenge.
This is avoided in our setting by our choice of boundary conditions.

\begin{lemma}
	For $(v,\mu) \in H^2(\go) \times \R$, one has
	\begin{align*}
		\left. \frac{1}{2}\frac{\ud^2 \mathcal{L}_s(v, \mu)}{\ud s^2} \right|_{s=0}=&
		\frac{\kappa}{2}\int_\go \left( (\Delta_\go v)^2 - \frac{2}{R^2}(\tau\cdot \nabla_\go u)^2 
		 + \frac{1}{R^4}v^2 \right) \dee \go
		\\
		&+\kappa\int_\go \left(  {c_0}{} \Delta_\go v + \frac{c_0^2}{2}\right) \dee \go
		+ \frac{\mu}{R} \int_\go v \dee \go
		- \mathcal{F}'(\go) [v\nu].
	\end{align*}
\end{lemma}
\begin{proof}
	It holds that
	\begin{align*}
		\left.\frac{\ud^2 \mathcal{L}(v,\mu)}{\ud s^2}\right|_{s=0}
		=&
		\mathcal{E}''(\go)[u\nu,v\nu] + \sigma_0 A''(\go)[v\nu,v\nu]
		-
		\kappa \left. \frac{\ud^2}{\ud s^2}\right|_{s=0} \left( s \int_{\Gamma_s} H c_0\circ \pi \dee \Gamma \right)
		\\
		&+
		\frac{\kappa}{2}\left. \frac{\ud^2}{\ud s^2}\right|_{s=0} \left( s^2 \int_{\Gamma_s} c_0^2 \circ \pi \dee \Gamma  \right)
		+ 2\mu A'(\go)[v\nu]
		- 2 \mathcal{F}'(\go)[v\nu].
	\end{align*}
	The second variations of $W$ and $A$ are given in Proposition \ref{prop:secondVar},
	For the spontaneous curvature terms, one has
	\[
	 \left.\frac{\ud^2}{\ud s^2}\right|_{s=0}  \left( -\kappa s \int_{\Gamma_s} H c_0 \circ \pi \dee \Gamma  + \frac{\kappa s^2}{2} \int_{\Gamma_s} c_0^2 \circ \pi \dee \Gamma \right)
	 =
	 \kappa \int_\go\left( 2c_0 \Delta_\go v + c_0^2 \right)\!\! \dee \go.
	\]
	Combining the calculations, we arrive at
	\begin{align*}
		\left.\frac{\ud^2 \mathcal{L}_s(v, \mu)}{\ud s^2} \right|_{s=0}=&
		\kappa \int_\go \left( (\Delta_\go v)^2 - \frac{2}{R^2} (\tau\cdot \nabla_\go v)^2 - \frac{1}{2R^2}|\nabla_\go v|^2 + \frac{1}{R^4} v^2 \right) \dee \go
		\\
		&+
		\sigma_0 \int_\go |\nabla_\go v|^2 \dee \go
		+
		\kappa \int_\go \left( 2\Delta_\go v c_0 + c_0^2 \right) \dee \go
		\\
		&+
		\frac{2\mu}{R} \int_\go v \dee \go
		 - 2 \mathcal{F}'(\go)[v\nu].
\end{align*}
Recalling that $\sigma_0 = \frac{\kappa}{2R^2}$ completes the result.
\end{proof}
We note that $\tau \cdot \nabla_\go v = \frac{1}{R}\partial_\theta ( v \circ X)\circ X^{-1}$ is an anisotropic term which represents the derivative in the curved direction.
For convenience, we define $L\colon H^2(\go)\times \R^2\to \R$ by
	\[
		L(v,\mu) := \left.\frac{1}{2}\frac{\ud^2 \mathcal{L}_s(v,\mu)}{\ud s^2}\right|_{s=0}.
	\]

\section{Critical points of the approximating Lagrangian}
\label{sec:critPoints}
From the calculations above, we have demonstrated that
\begin{equation}
	\mathcal{L}_{\rho} \left( v, \mu \right) = C_0 + \rho C_1 + \rho^2 L(v, \mu) + O(\rho^3)
\end{equation}
where $C_0 = \kappa W(\go)$ and $C_1 = -\frac{\kappa}{R} \int_\go c_0 \dee \go$ are independent of $\rho \in \R$, $v\in H^2_0(\go)$, and $\mu \in \R$.
Since $L$ is quadratic, it is much easier to analyse than the highly non-linear functional $\mathcal{L}_{\rho}$.
Hence, we seek critical points of $L$.
It is convenient to introduce the following function space.
\begin{definition}
Define $U_C$ by
\[
	U_C := \left\{ v \in H^2_0(\go) : \int_\go v \!\! \dee \go = 0  \right\},
\]
the space of admissible membranes with clamped boundary.
\end{definition}

We wish to define a bilinear form so that we can apply standard Lax-Milgram-type results for the minimisation of quadratic energies under constraints.

\begin{definition}\label{def:BilinearFormsA}
	We define the bilinear form $a\colon H^2(\go) \times H^2(\go) \to \R$ by
	\begin{align*}
		a(v,\eta):=&
		\kappa\int_{\go} \left( \Delta_\go v \Delta_\go \eta - \frac{2}{R^2}(\tau\cdot \nabla_\go v) (\tau\cdot \nabla_\go \eta) + \frac{1}{R^4} v \eta \right) \dee \go,
	\end{align*}
	the quadratic functional $J\colon H^2(\go) \to \R$ by
	\[
		J(v):= \frac{1}{2}a(v,v),
	\]
	and the linear form $\ell\colon H^2(\go) \to \R$ by
	\[
		\langle \ell,v\rangle := \kappa \int_\go - c_0 \Delta_\go v \dee \go
	\]
\end{definition}
Notice that for $v \in H^2(\go) \cap H_0^1(\go)$ and $\mu \in \R$,
\[
	L(v,\mu) = J(v) + \frac{\kappa}{2}\int_{\go} c_0^2 \dee \go + \frac{\mu}{R} \int_\go v \dee \go - \langle f + \ell, v\rangle
\]
which follows from the assumption on the form of $\mathcal{F}'(\go)$ made in Assumption \ref{ass:FGo0}.
We may now state the problem of interest.
\begin{problem}\label{prob:ClampedBoudnary}
	Find $u \in U_C$ such that
	\[
		a(u,v) = \langle f + \ell, v\rangle \quad \forall v \in U_C.
	\]
	Equivalently, find $u \in U_C$ such that
	\[
		J(u) - \langle f+\ell, u\rangle = \min_{v \in U_C} \left( J(v) - \langle f + \ell, v\rangle \right).
	\]
\end{problem}

\begin{remark}
    A solution to Problem \ref{prob:ClampedBoudnary} will (weakly) satisfy the following PDE
    \begin{equation}\label{eq:ClampedPDEForMinimisation}
    \begin{split}
        \kappa\left( \Delta_\go^2 u + \frac{2}{R^2}\Div_\go \left(\tau \tau^T \nabla_\go u \right)+ \frac{1}{R^4}u \right)&=f - \kappa \Delta_\go c_0 - \frac{\mu}{R} \quad \mbox{in}~ \go,
        \\
        u|_{\partial\go} =
        \conormal \cdot \nabla_\go u |_{\partial\go} &= 0,
        \\
        \int_\go u \dee \go &= 0,
    \end{split}
    \end{equation}
    where $\mu\in \R$ is the Lagrange multiplier enforcing $\int_\go u \dee \go =0$.

    Notice that formally, as $R \to \infty$, we obtain the PDE $\kappa \Delta^2 u = f - \kappa\Delta c_0 $, recovering the Monge-Gauge \cite{EllGraHob16} with no surface tension.
\end{remark}

\subsection{Coercivity of the quadratic model}

To show well-posedness of Problem \ref{prob:ClampedBoudnary} it is sufficient to show that $a$ is coercive in the $\|\cdot\|_{2,2}$ norm over the larger space $H^2(\go)\cap H_0^1(\go)$ since it is clearly bounded and bilinear.
Indeed it is clear that $a$ is non-negative over $H^2(\go) \cap H^1_0(\go)$ since
\[
	a(v,v) \geq \kappa\int_\go \left(\Delta_\go v + \frac{1}{R^2} v\right)^2 \dee \go
\]
for any $v \in H^2(\go)\cap H^1_0(\go)$.

We now give the following spectral result for the Laplace-Beltrami operator on $\go$.
\begin{proposition}
	The eigenvalues of $-\Delta_\go$ in $H^2(\go) \cap H^1_0(\go)$ are given by
	\[
		\left\{ \left( \frac{m\pi}{L}\right)^2 + \left(\frac{n}{R} \right)^2 : m,n \in \mathbb{N}_0,\, m>0 \right\},
	\]
	with eigenfunctions (in polar coordinates),
	\[
		\left\{
		(z,\theta) \mapsto
		\begin{cases}
		\cos\left(\frac{m\pi}{L} z\right) \cos(n\theta),\, \cos\left(\frac{m\pi}{L} z\right) \sin(n\theta) \quad m \mbox{ odd}
		\\
		\sin\left(\frac{m\pi}{L} z\right) \cos(n\theta),\, \sin\left(\frac{m\pi}{L} z\right) \sin(n\theta) \quad m \mbox{ even}
		\end{cases}
		: m,n \in \mathbb{N}_0, m>0 \right\}.
	\]
\end{proposition}
\begin{proof}
This is a consequence of $\go$ being locally isometrically-isomorphic to a rectangular subset of a plane and standard arguments for the eigenvalues/functions of the Laplace operator on a rectangle.
An immediate consequence of this is that the span of the eigenfunctions is dense in $H^2(\go) \cap H^1_0(\go)$.
\end{proof}
\begin{theorem}\label{thm:EnergyCoercive}
	Let $\alpha := \frac{\pi R}{L}$ be half of the aspect ratio of $\go$.
	For $\alpha>0$, there is $c(\alpha)>0$ such that
	\[
		a(v,v) \geq c(\alpha) \|v\|_{2,2}^2
	\]
	for any $v \in H^2(\go) \cap H^1_0(\go)$.
\end{theorem}
\begin{proof}
    Notice that, if $P_{\tilde{N}} \colon H^2(\go)\cap H^1_0(\go) \to H^2(\go)\cap H^1_0(\go)$ is the projection onto the first $\tilde N$ eigenfunctions, then $a(u, P_{\tilde N} v) = a(P_{\tilde N} u,v)$ for all $u,v \in H^2(\go) \cap H^1_0(\go)$ by the particular form of $a$.
    Therefore, one may calculate
    \[
        a(P_{\tilde N} u,P_{\tilde N} u) \leq a(P_{\tilde N} u,P_{\tilde N} u) + a(u-P_{\tilde N}u,u-P_{\tilde N} u) = a(u,u).
    \]
    Since $P_{\tilde N} v$ converges weakly in $H^2(\go) \cap H^1_0(\go)$ to $v$ as $\tilde N \to \infty$, it is sufficient to assume that $v$ is given by
	\[
		\left (v \circ X\right)(z,\theta) = \sum_{m=1}^M \sum_{n=0}^N \alpha_{mn}\cos\left( \frac{m\pi}{L}z\right) \sin \left( n(\theta-\theta_{mn}) \right)
		+
		\beta_{mn}\sin\left( \frac{m\pi}{L}z\right) \sin \left( n(\theta-\Theta_{mn}) \right),
	\]
	where $\alpha_{mn} = 0$ for $m$ even and $\beta_{mn} = 0$ for $m$ odd.
	It is then possible to calculate
	\[
		a(v,v)
		=
		\kappa c\sum_{m=1}^M\sum_{n=0}^N (\alpha_{mn}^2+\beta_{mn}^2)
		\left[ \left(\left( \frac{n}{R}\right)^2 + \left(\frac{m\pi}{L} \right)^2 \right)^2 -\frac{2}{R^2}\left( \frac{n}{R}\right)^2 + \frac{1}{R^4}\right],
	\]
	where $c$ is the normalising constant such that
	\[
		\int_{\mathbb{S}^1 \times \left( -\frac{L}{2},\frac{L}{2}\right)} \cos\left( \frac{m\pi}{L}z\right)\sin(n\theta)\cos\left( \frac{m'\pi}{L}z\right)\sin(n'\theta) \dee z \dee \theta =c \delta_{m,m'} \delta_{n,n'}.
	\]
	We now wish to estimate the summand from below.
	If $\alpha \geq 1$,
	\begin{align*}
	\left(\left( \frac{n}{R}\right)^2 + \left(\frac{m\pi}{L} \right)^2 \right)^2 -\frac{2}{R^2}\left( \frac{n}{R}\right)^2 + \frac{1}{R^4}
	\geq
	\frac{1}{2}	\left(\left( \frac{n}{R}\right)^2 + \left(\frac{m\pi}{L} \right)^2 \right)^2 + \frac{1}{R^4},
	\end{align*}
	therefore
	\[
		a(v,v) \geq \frac{\kappa}{2}\|\Delta_\go v\|_{0,2}^2 + \frac{\kappa}{R^4}\|v\|_{0,2}^2.
	\]
	For $\alpha<1$ we see that
	\begin{align*}
	\left(\left( \frac{n}{R}\right)^2 + \left(\frac{m\pi}{L} \right)^2 \right)^2 -\frac{2}{R^2}\left( \frac{n}{R}\right)^2 + \frac{1}{R^4}
	\geq&\,
	\alpha^2 \left( \frac{n}{R}\right)^4 + \left(\frac{m\pi}{L} \right)^4 + 2\left( \frac{n}{R}\right)^2\left(\left( \frac{m\pi}{L}\right)^2 - \frac{\alpha^2}{R^2} \right) + \frac{\alpha^2}{R^4}
	\\
	\geq&\,
	\frac{\alpha^2}{2} \left( \left( \frac{n}{R}\right)^2 + \left(\frac{m\pi}{L} \right)^2\right)^2  + \frac{\alpha^2}{R^4}.
	\end{align*}
	This has therefore shown that
	\[
		a(v,v) \geq \frac{\kappa\alpha^2}{2}\|\Delta_\go v\|_{0,2}^2 + \frac{\kappa \alpha^2}{R^4}\|v\|_{0,2}^2.
	\]
\end{proof}

\begin{remark}
Notice that the coercivity constant $c(\alpha)$ depends on the aspect ratio $\alpha$ with $c(\alpha) \to 0$ when $\alpha\to 0$.
It is possible to see that this decay is in some sense sharp.
Fix $R = 1$, $\kappa = 1$ and select $u\circ X =  \frac{1}{\sqrt{c}} \cos\left(\frac{\pi}{L} z\right)\sin(\theta)$, where $c$ is a normalising factor, then it is seen that $u \in H^2(\go) \cap H_0^1(\go)$ satisfies
\[
    a(u,u) = \alpha^4 + 2\alpha^2.
\]
which indeed decays as $\alpha\to 0$.
\end{remark}

We now have the following existence and regularity result.
\begin{theorem}\label{thm:regularity}
	Let $c_0 \in L^2(\go)$.
	There exists a unique solution $u \in U_C$ to Problem \ref{prob:ClampedBoudnary}.
	If $c_0 \in H^2(\go)$ and there is $q>2$ such that $f \in \left( W^{1,q}_0(\go) \right)^*$, then it holds that $u\in W^{3,p}(\go)$ where $p<2$ satisfies $\frac{1}{p} + \frac{1}{q} = 1$.
	Further to this, if $c_0 \in H^2(\go)$ and $f \in L^2(\go)$ it holds that $u \in H^4(\go)$.
\end{theorem}
\begin{proof}
We see that $a$ is bounded, coercive and bilinear.
By Lax-Milgram the existence of a solution follows.
The first regularity result is effectively a minor adaptation of Proposition B.2 of \cite{EllHer21}.
The second regularity result follows from the smoothness of the boundary and standard Biharmonic regularity theory, for example \cite[Theorem 2.20]{GazGruSwe10}.
\end{proof}
Let us note that the assumption that $f \in \left(W^{1,q}_0(\go)\right)^*$ for $q>2$ allows for the case that $f$ is a (signed) measure, i.e.~in the dual of continuous functions.

\section{Example causes of small deformations}\label{sec:ExampleForcing}
We now provide different examples of the causes of small deformations and related calculations.
\subsection{Point forces}\label{sec:pointForce}
Here, one might be interested in the case that a point force is being applied to the membrane, this may be due to, for example, optical tweezers \cite{Nus17}.
If one considers $\mathcal{F}$ to be given by
\begin{equation}
    \mathcal{F} (\Gamma) = -\sum_{i=1}^N \beta_i d_\Gamma (\XX_i),
\end{equation}
where $d_\Gamma$ is the oriented distance function for $\Gamma$ and $N \in \mathbb{N}$, $\beta \in \R^N$, and $\XX_i \in \go$ for $i = 1,\ldots, N$ are given data.
Notice that when evaluating $\mathcal{F}$ at $\Gamma(\rho u)$, then up to leading order $\mathcal{F}(\Gamma(\rho u)) \approx \sum_{i=1}^N \beta_i \rho u(\XX_i)$.
It is then possible to calculate $f = \sum_{i = 1}^N \beta_i\delta_{\XX_i}$ where $\delta_X$ is the Dirac delta functional at $X\in \go$.
Point forces have been considered in \cite{EllFriHob17} on a sphere, where one may also consider the problem of optimising the location of $\XX_i$, $i=1,\dots,N$, see \cite[Proposition 5.1]{EllFriHob17}.
We remark that for our situation with a two dimensional surface, this choice of $f$ satisfies $f \in \left( W^{1,q}_0(\go) \right)^*$ for any $q > 2$.

\subsection{Point constraints}\label{sec:examples:pointConstraints}
One might be interested that a protein is attached to a biomembrane at a one or more points, inducing point constraints; this is discussed in \cite{GraKie17} and \cite{EllHer21}, as well as associated membrane-mediated interactions between groups of points.
Mathematically, point constraints were incorporated through a penalty in Problem 5.3 of \cite{EllFriHob17} and the numerical analysis of such a problem was considered within \cite{Her20,EllHer20}.
Indeed, the penalisation of point constraints may be considered a model in its own right, this corresponds to a situation whereby particles are attached to the membrane via springs, however we do not discuss it here.
In order to include point constraints into the full non-linear model without using a penalty approach, we consider the Lagrangian
\[
    \mathcal{L}_\rho^{\rm points}(u,\mu,\lambda) = \mathcal{L}_\rho (u,\mu) + \sum_{i=1}^N (\beta_i + \rho \lambda_i) d_{\Gamma_\rho } (\XX_i(\rho)),
\]
where
\[
    \XX_i(\rho) := \XX_i + \rho Z_i \nu(\XX_i)\mbox{ for }i=1,\dots,N
\]
is a collection of points which are sufficiently near $\go$ and $\beta \in \R^N$.
One may see that the extra term we have added onto $\mathcal{L}_\rho$ corresponds to the Lagrange multiplier of the constraint that, for a critical $u$, $\XX_i(\rho) \in \Gamma_\rho$ for $i = 1,\dots,N$.

When proceeding as in Lemma \ref{lem:FirstDerivative}, calculating the first derivative of this extra term, one may see that, similarly to \cite[Problem 5.3]{EllFriHob17},
\begin{equation}
	\left.\frac{\ud}{\ud s} \right|_{s=0} \left( (\beta_i + s \lambda_i) d_{\Gamma_s}(\XX_i(s)) \right) = \beta_i (Z_i - v(\XX_i)).
\end{equation}
To ensure that the first order terms are vanishing, one expands around $\beta = 0$, so that
\[
	\left. \frac{\ud \mathcal{L}_s ^{\rm points}\left(u,\mu,\lambda \right)}{\ud s} \right|_{s=0} = -\frac{\kappa}{R}\int_\go c_0 \!\! \dee \go.
\]

By examining the second derivative,
\begin{equation}
	\left.\frac{\ud^2}{\ud s^2}\right|_{s=0} \left(  s \lambda_i d_{\Gamma_s}(\XX_i(s)) \right) = 2 \lambda_i (Z_i - u\XX_i)),
\end{equation}
we see that the quadratic energy $J_{\rm points}$ associated to the Lagrangian $\mathcal{L}_{\rm points}$ is $J$.
That is to say there is no contribution to the small deformation energy $J$ from point constraints.
There is however a change in the quadratic Lagrangian $L_{\rm points}$.
This change in the Lagrangian means that any critical point satisfies $u(\XX_i)=Z_i$ for $i=1,\dots,N$.

The following existence result considers point constraints near the surface.
The existence follows by coercivity shown in Theorem \ref{thm:EnergyCoercive}, and the regularity follows from the regularity shown from Theorem \ref{thm:regularity}.
\begin{corollary}\label{cor:PointConstraint}
	Let $\{\XX_i\}_{i=1}^N \subset \go$ be a pairwise disjoint collection of points, let $\{Z_i\}_{i=1}^N\subset\R$.
	There is unique $u \in U_C$ such that $u(\XX_i) = Z_i$ for $i=1,\dots,N$ and
	\[
		J(u) - \langle f+\ell, u\rangle = \inf \{ J(v) - \langle f+\ell, v\rangle : v \in U_C ,\, v(\XX_i) = Z_i,\, i=1,\dots,N\}.
	\]
	If $c_0 \in H^2(\go)$ and $f \in \left( W^{1,q}_0(\go)\right)^*$ for some $q>2$, it further holds that $u \in W^{3,p}(\go)$ for $p<2$ satisfying $\frac{1}{p}+ \frac{1}{q} = 1$.
\end{corollary}

\subsection{Forcing along a curve}\label{sec:forceOnALine}
In Section \ref{sec:pointForce}, where point point forcing was considered, there is the possibility of many points.
It might be the case that many points are more efficiently modelled as one or multiple curves.
This modelling approach may be relevant for protein assemblies that organise along curves on the membrane, such as BAR-domain scaffolds or dynamin collars, where the collective mechanical effect is distributed along a one-dimensional structure embedded in the surface.
Such protein rings have been described on tubules, see \cite[Section 4]{Gov18}, for example.

For ease of exposition, let us consider only a single curve.
In such a case, the sum becomes an integral and one considers
\begin{equation}
	\mathcal{F}(\Gamma) = -\int_\gamma d(\Gamma,\cdot) w \dee \gamma,
\end{equation}
where $d$ is again the oriented distance function, $\gamma\subset \go$ is some sufficiently smooth curve, and $w\colon \gamma \to \R$ is a sufficiently smooth \emph{weight} function.
With this in mind, one may calculate
\[
  \langle f, v \rangle = \int_\gamma v w \dee \gamma.
\]
\begin{remark}\label{rem:regularityOfLineForcing}
Let us note that when $\gamma$ and $w$ are appropriately smooth, one can apply the regularity mentioned in the second part of Theorem \ref{thm:regularity}.
Depending on the smoothness of $w$ and $\gamma$, the solution may be more regular, since the force $f$ may more regular than the case which admits Dirac masses.
\end{remark}

\subsection{Phase-field forcing}
Motivated in part by the work of \cite{Hat20}, we consider the effect of the concentration of saturated lipids which induce a spontaneous curvature, $H_s = \rho c_0\circ \pi$, where $c_0$ is a function of a concentration $\phi\colon \go \to \R$.
It is observed that certain lipid mixtures undergo phase separation, forming so-called lipid rafts \cite{SimIko97}.
This may be modelled by coupling the elastic energy to a concentration dependent quantity, which first appears in \cite{Lei86}, and has been considered in many subsequent works, e.g., \cite{BraLusSte20, EllSti10-A,EllSti13,BarGarNur18}.
In the settings of small deformations around a sphere, such models have been investigated in \cite{EllHatHer20,EllHat21,EllHatSti20, CaeEllGra26}.
Let us also note the works \cite{Gov18, SchGov08} which also consider a concentration dependent forcing.

The model considered in the works \cite{EllHatHer20,EllHat21,EllHatSti20,CaeEllGra26} supposes that there should be some phase separation which involves adding a Ginzburg-Landau energy.
Because of the scaling we will shortly mention, let us add the extra term to the energy, rather than considering it as part of $\mathcal{F}$.
That is to say, consider the Lagrangian
\[
	\mathcal{L}^{\rm pf}_\rho(u,\mu, \phi)
	:=
	\mathcal{L}_\rho(u, \mu)
	+
    \rho^2 b \int_{\Gamma_\rho} \left( \frac{\epsilon}{2}|\nabla_{\Gamma_\rho} \phi|^2\circ \pi + \frac{1}{\epsilon} W(\phi)\circ \pi \right) \!\! \dee \Gamma.
\]
where $b>0$ is a line tension term, $\epsilon>0$ is a small parameter and $W \colon \R \to \R$ is a double-well function.
We note that this has a $\rho$ term included so that the forcing scales like $\rho^2$ which is motivated by \cite{KuzAkiChi05} where, for a monolayer, the authors calculate that the line tension between phases (so-called raft and non-raft regions) scales quadratically to the so-called hydrophobic mismatch which is $\rho$ in this case.
Calculations within \cite{EllHat21} show that these term are approximated by
\[
	\rho^2 b \int_{\Gamma_\rho} \left( \frac{\epsilon}{2}|\nabla_{\Gamma_\rho} \phi|^2\circ \pi + \frac{1}{\epsilon} W(\phi)\circ \pi \right)\!\! \dee \Gamma
	=
	\rho^2 b \int_{\go} \left( \frac{\epsilon}{2}|\nabla_\go \phi|^2 + \frac{1}{\epsilon} W(\phi) \right)\!\! \dee \go
    + O(\rho^3).
\]
This leads to the following small deformations energy
\begin{equation}
\begin{split}
	J_{\rm pf}(v,\phi) :=& \frac{\kappa}{2} \int_\go \left( (\Delta_\go v)^2 - \frac{2}{R^2}(\tau\cdot \nabla_\go v)^2 + \frac{1}{R^4} v^2 \right)\!\! \dee \go
	\\
	&+
	b \int_\go \left( \frac{\epsilon}{2}|\nabla_\go \phi|^2 + \frac{1}{\epsilon}W(\phi) \right) \!\! \dee \go
	+
	\kappa \int_\go \left( \Delta_\go v c_0(\phi) + \frac{c_0^2(\phi) }{2} \right)\!\! \dee \go
	- \mathcal{F}'(\go)[v\nu].
\end{split}
\end{equation}
For a fixed $\phi$ such that $c_0(\phi)$ is sufficiently regular, it is possible to verify that there exist minimisers of the map $v \mapsto J_{\rm pf}(v,\phi)$ over $v \in U_C$.
The interesting aspect of adding this phase field energy is to minimise the coupled energy over some appropriate function space such such that $\int_\go c_0(\phi)\dee\go$ is constant -- where we recall that this is a condition required for the first order terms to be neglected.
Such minimisation may be done with gradient flow dynamics as in \cite{EllHer20,EllHat21,EllHatSti20}; if $W$ is non-smooth, say the obstacle potential \cite{BloEll91,BloEll92,OonPur88} and if one is not concerned with dynamics and seeks only to find some form of minimiser, one may wish to use the so-called \emph{variable metric projection theorem} (VMPT) approach \cite{BlaRup17}.

\section{Numerical examples}\label{sec:numExp}
Here we consider a few numerical examples of the forcing we mentioned.
We do not discuss the numerical analysis of the schemes used, this is left to future work; we do however, verify that convergence is obtained for a problem with a point force.
The experiments are conducted making use of the DUNE framework~\cite{BasBlaDed21,DUNE}, in particular the Python bindings of the DUNE-FEM module~\cite{DUNE-python},
the results are visualised using Paraview \cite{Paraview}.

\subsection{Numerical scheme for the fourth order PDE on a surface}
We are required to introduce number of standard notions from the theory of surface finite elements.
We refer the reader to \cite{EllRan21} for further details.
The tube will be approximated by a triangulated surface $\goh$.
All vertices of this triangulated surface should lie on $\go$.
Let us note that the boundary will be approximated, that is to say $\partial \goh \neq \partial \go$, however, we will have that the projection of the boundary satisfies $\pi \left( \partial \goh\right) = \partial \go$.
We denote the triangulation by $\mathcal{T}_h$, where $h = \sup_{T \in \mathcal{T}_h} {\rm diam}(T)$.
On this triangulated surface we consider the following family of spaces, for $p \in \mathbb{N}$ with $p\geq 2$:
\begin{equation}
	V_h^p : = \{ v_h \in C^0_0(\goh) : v|_T \in \mathbb{P}^p(T),\, T \in \mathcal{T}_h\},
\end{equation}
where $\mathbb{P}^p(T)$ denotes polynomials of order less than $p$ on $T$.
This is a $H^1$ conforming space but not a $H^2$ conforming space.
It is worth mentioning that the question of a $H^2$, or $C^1$, conforming space is particularly challenging on a triangulated surface, as the discrete surface is not itself a {$C^1$} approximation of the surface.

Our strategy for the discretisation will utilise the so-called interior penalty $C^0$ method \cite{BreSun05}.
To do this on surfaces, we refer the reader to \cite{LarLar17}, where a theory is presented in the case of quadratic elements.
We now provide a few necessary details relating to the method before giving a discrete energy.
On an edge $e := \partial T^+ \cap \partial T^-$, the jump and average are given at $x \in e$ by
\begin{equation}\begin{split}
    \jump{u_h} (x) :=&\ \lim_{s\to 0^+} \left( u_h(x- s \conormal_{T^+}) - u_h(x- s \conormal_{T^-})  \right),
    \\
    \avg{u_h} (x) :=&\ \lim_{s\to 0^+} \frac{1}{2} \left( u_h(x- s \conormal_{ T^+}) + u_h(x- s \conormal_{ T^-})  \right),
\end{split}\end{equation}
where $\conormal_{T}$ is the unit outwards co-normal to $T$.
For boundary edges, that is $e:= \partial T \cap \partial \goh$,
these are provided by
\begin{equation}\begin{split}
    \jump{u_h}(x) :=&\ \lim_{s\to 0^+}  u(x- s \conormal_T),
    \\
    \avg{u_h}(x) :=&\ \lim_{s\to 0^+} u(x- s \conormal_T).
\end{split}\end{equation}

Let $\mathcal{E}_h$ denote the set of edges of $\mathcal{T}_h$.
The discrete $H^2$-type norm is given by
\begin{equation}
    |\!|\!| u_h |\!|\!|_{\goh}
    :=
    \left(
        \sum_{T \in \mathcal{T}_h} \| \Delta_\goh u_h\|_{L^2(\goh)}^2 + \sum_{e \in \mathcal{E}_h} |e|^{-1} \| \jump{ \conormal \cdot \nabla_\goh u_h } \|_{L^2(e)}^2
        +
        \|u_h\|_{H^1(\goh)}^2
        \right)^{\frac{1}{2}}
\end{equation}
for any $u_h \in V_h^p$, where discrete $H^1$ and $L^2$ norms are standard since $V_h^p$ is $H^1$ conforming.

The discrete form of $a$, $a_h \colon V_h^p \times V_h^p \to \R$ is given by
\begin{equation}\begin{split}
    a_h(u_h,v_h)
    :=&\
    \sum_{T \in \mathcal{T} } \int_T \left( \Delta_\goh u_h \Delta_\goh v_h - \frac{2}{R^2} \left(\tau \cdot \nabla_\goh u_h \right) \left(\tau \cdot \nabla_\goh v_h \right) + \frac{1}{R^4} u v\right) \dee x
    \\
    &+
    \sum_{e \in \mathbb{E}} \int_{e} \left( -\avg{\Delta_\goh u_h} \jump{\nabla_\goh v_h \cdot \conormal } -\avg{\Delta_\goh v_h} \jump{\nabla_\goh u_h \cdot \conormal } + \beta \jump{\nabla_\goh u_h \cdot \conormal} \ \jump{\nabla_\goh v_h \cdot \conormal} \right)\dee s,
\end{split}\end{equation}
where $\beta \colon \bigcup_{e \in \mathcal{E}_h} e \to \R$ is positive and plays the role of a penalty.
In this work we take it as $\beta|_{e} := 30 p^2 |e|^{-1}$.
In order to handle the constraint on the integral of the height function, we will make use of a Lagrange multiplier.

\subsection{Verification of numerical scheme}\label{sec:numericalVerification}
Here, we verify that the numerical scheme is converging as expected for an example with point forcing; this is chosen due to its low regularity, hence it is expected to have the worst approximation.
Such an example is resulting from a manufactured example, whereby we choose the exact solution $u$ by its parametrisation to make use of software which provides derivatives.
We set
\begin{equation*}
    u\circ X(\theta,z)
    =
    \left( C + (\theta^2 + z^2)\log(\theta^2 + z^2) \right) \frac{(\theta^2 - \pi^2)^4}{\pi^8} \left(z^2 - \left(\frac{L}{2}\right)^2\right)^2,
\end{equation*}
where $C = -\frac{133693875}{8192 \pi^9}$ is a constant which ensures that $\int_\go u \dee \go = 0$.
Here, we comment that $\frac{1}{16 \pi} (\theta^2 + z^2)\log(\theta^2 + z^2)$ is the fundamental solution of the bi-laplace problem on $\R^2$, however in a periodic setting, it has a discontinuous derivative at $\theta = \pm \pi$, which is dealt with via $\frac{(\theta^2 - \pi^2)^4}{\pi^8}$.
Furthermore, $(z^2 - (L/2)^2)^2$ is included to correct for the clamped boundary conditions at $z = \pm \frac{L}{2}$.
This leads to
\begin{equation*}
    -\Delta_\go^2 u + \frac{2}{R^2}\Divg(\tau \otimes \tau \nabla_\go u) + \frac{1}{R^4}u
    =
    16 \pi \delta_{ (0,1,0) } + g
\end{equation*}
where $g\in L^2(\Gamma)$ may be given by a lengthy calculation.
The results of \cite{LarLar17} appears to make use of $H^3$ regularity of the solution; here we have that $u \in W^{3,p}(\go)$ for any $p \in [1,2)$, as such, this is borderline out of scope for such a convergence result.
We expect that one can extend the theory to cover such a case, but we do not investigate it in this work.
Using the above as data for a numerical experiment with $L = 2$, $R=1$, we obtain the results in Figure \ref{fig:convergenceExample}, which show the expected orders of convergence.
\begin{figure}
    \begin{tikzpicture}
        \begin{axis}[xmode = log, ymode = log, log basis x={2}, log basis y={2}, xlabel = $h$, ylabel = {Error}, legend pos=outer north east]
            \addplot table [x=h, y=H2, col sep=comma] {dirac_convergence.csv};
            \addlegendentry{$H^2$ seminorm}
            \addplot table [x=h, y=L2, col sep=comma] {dirac_convergence.csv};
            \addlegendentry{$L^2$ norm}
            \addplot[no marks, dashed] table [x=h, y expr = 4*\thisrow{h}, col sep=comma] {dirac_convergence.csv};
            \addlegendentry{$4 h$}
            \addplot[no marks, dotted] table [x=h, y expr = \thisrow{h}^2, col sep=comma] {dirac_convergence.csv};
            \addlegendentry{$h^2$}
        \end{axis}
    \end{tikzpicture}
    \caption{A plot of the error from the experiment in Section \ref{sec:numericalVerification}.
    We see that one appears to obtain the convergence as in \cite{LarLar17}, despite low regularity data.}\label{fig:convergenceExample}
\end{figure}
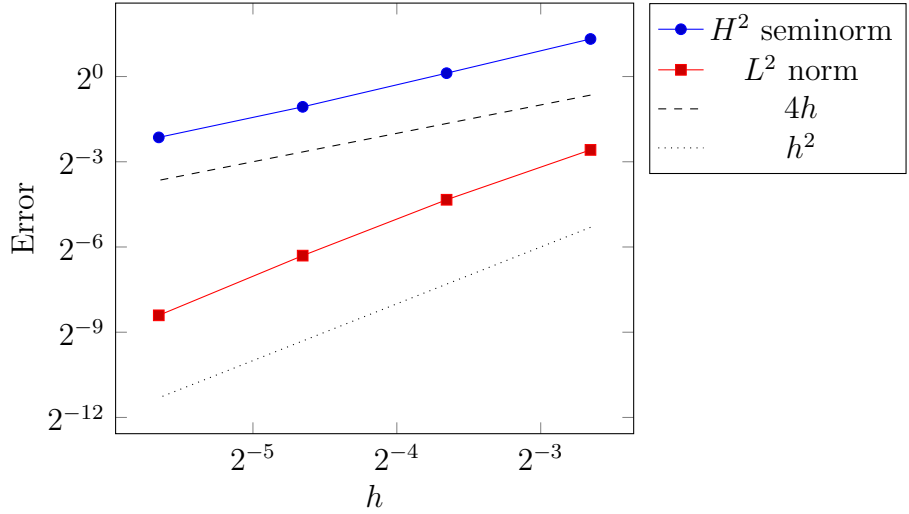

\subsection{Point Forcing and constraints}
We now turn to numerical examples with point constraints or point forcing.
\subsubsection{Forcing}\label{sec:exp:pointForce}
Here we set $L = 6$ and $R = 1$, as well as $f = \sum_{i = -3}^3 (-1)^i \delta_{\XX_i}$ where $\XX_i = (0,1, L i/4)$ for $i = -3,-2,\dots,3$.
This results in the picture in Figure \ref{fig:pointForce}
\begin{figure}
    \center
    \includegraphics[width = .8 \linewidth]{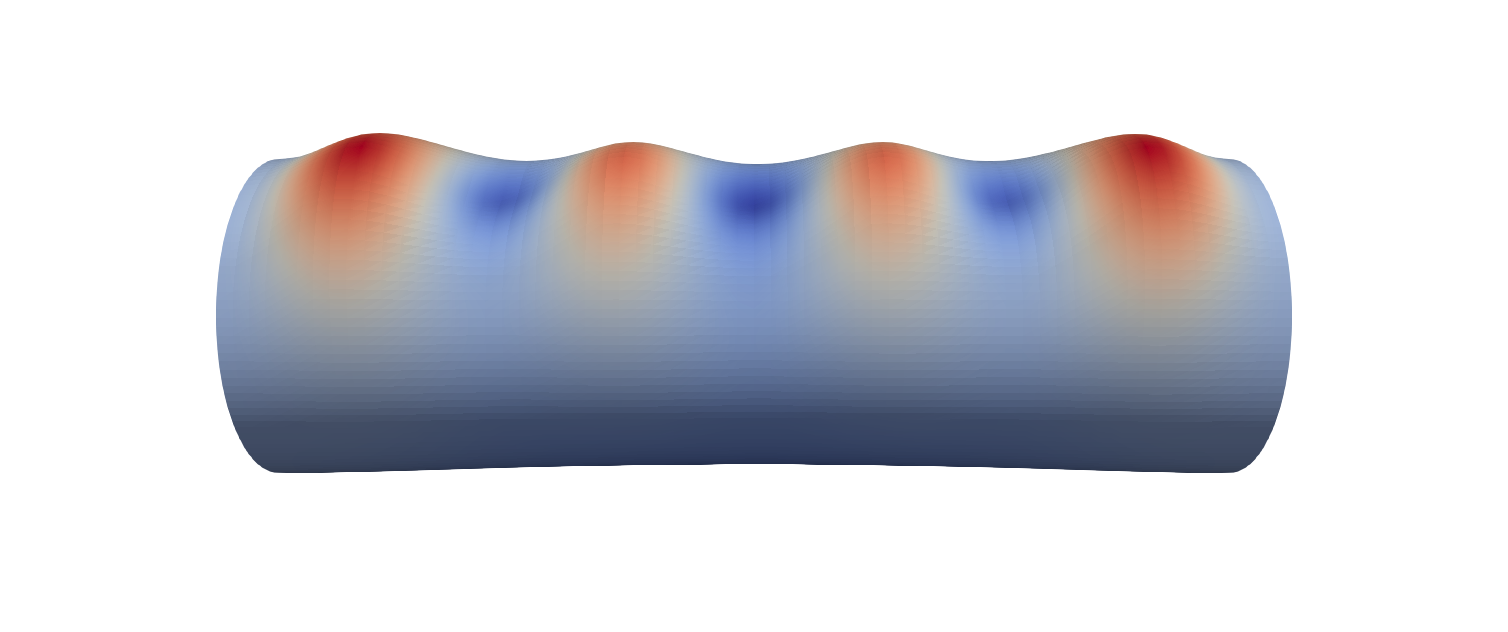}
    \caption{An image showing $\go(\rho u)$, where $u$ solves the point forcing problem outlined in Section \ref{sec:exp:pointForce}.
    For the purposes of visualisation, we set $\rho = 20$.
    Here, the colour corresponds to the value of $u$, with blue being low values and red being high values.}\label{fig:pointForce}
\end{figure}

\subsubsection{Constraints}\label{sec:exp:pointConstraints}
Here we consider the problem outlined in Section \ref{sec:examples:pointConstraints}.
We set $L = 6$ and $R = 1$, as well $f = 0$ and choose the point constraint locations as:
\begin{equation*}
    \{ x \in \go : x_2 \in \{-\frac{L}{6}, \frac{L}{6}\}, (x_1,x_2) \in \{ (-1,0), (1,0), (0,-1), (0,1) \}  \}
\end{equation*}
and set the constraint value constantly as $Z = 1$.
This results in the picture in Figure \ref{fig:pointConstraint}
\begin{figure}
    \center
    \includegraphics[width = .8 \linewidth]{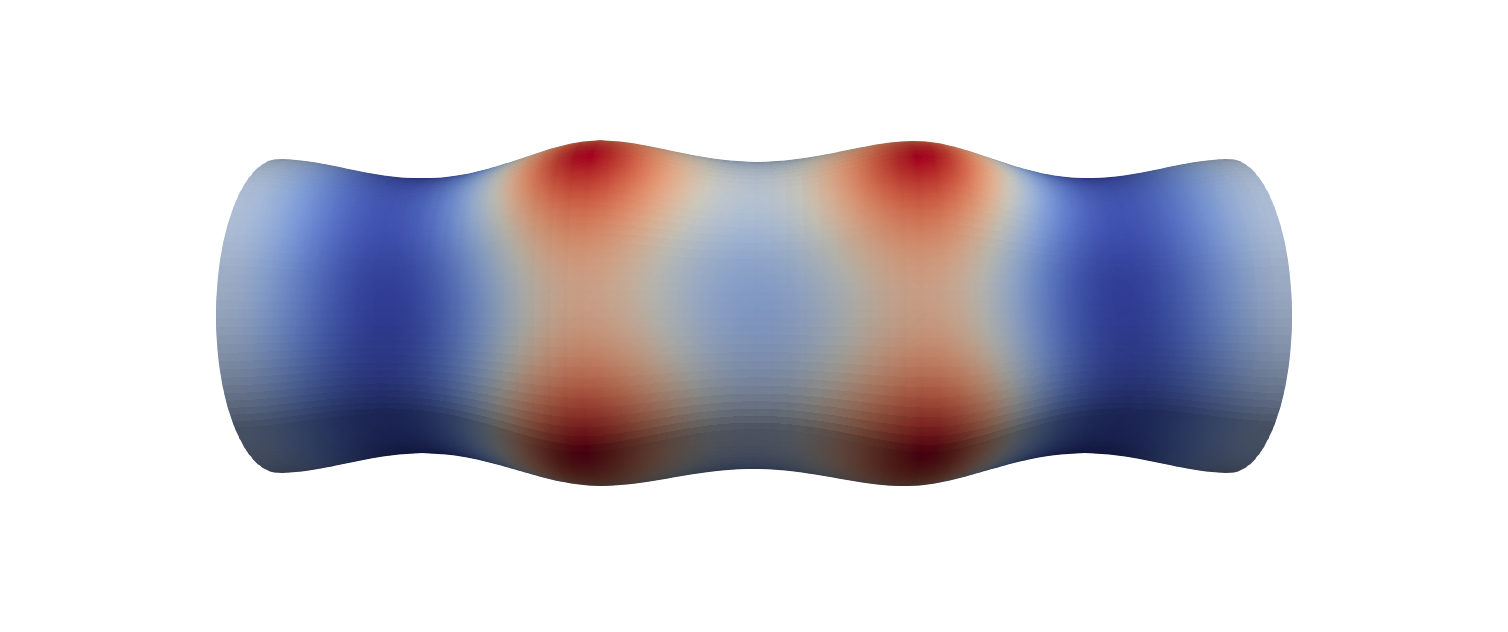}
    \caption{An image showing $\go(\rho u)$, where $u$ solves the point constraint problem outlined in Section \ref{sec:exp:pointConstraints}.
    For the purposes of visualisation, we set $\rho = 0.2$.
    Here, the colour corresponds to the value of $u$, with blue being low values and red being high values.}\label{fig:pointConstraint}
\end{figure}

\subsection{Line forcing}\label{sec:exp:lineForcing}
Here we set $L = 6$ and $R = 1$, and consider $ \gamma = \go \cap \{ x \in \R^3 : (x_2-R)^2 + (x_3^2 + x_1^2)^2 - x_3^2 - 10^{-6} = 0 \}$ and $f$ to be determined by $\langle f, v \rangle = \int_\gamma v \!\!\dee \gamma$.
This results in the picture in Figure \ref{fig:lineForcing}.
\begin{figure}
    \centering
    \includegraphics[width = .8 \linewidth]{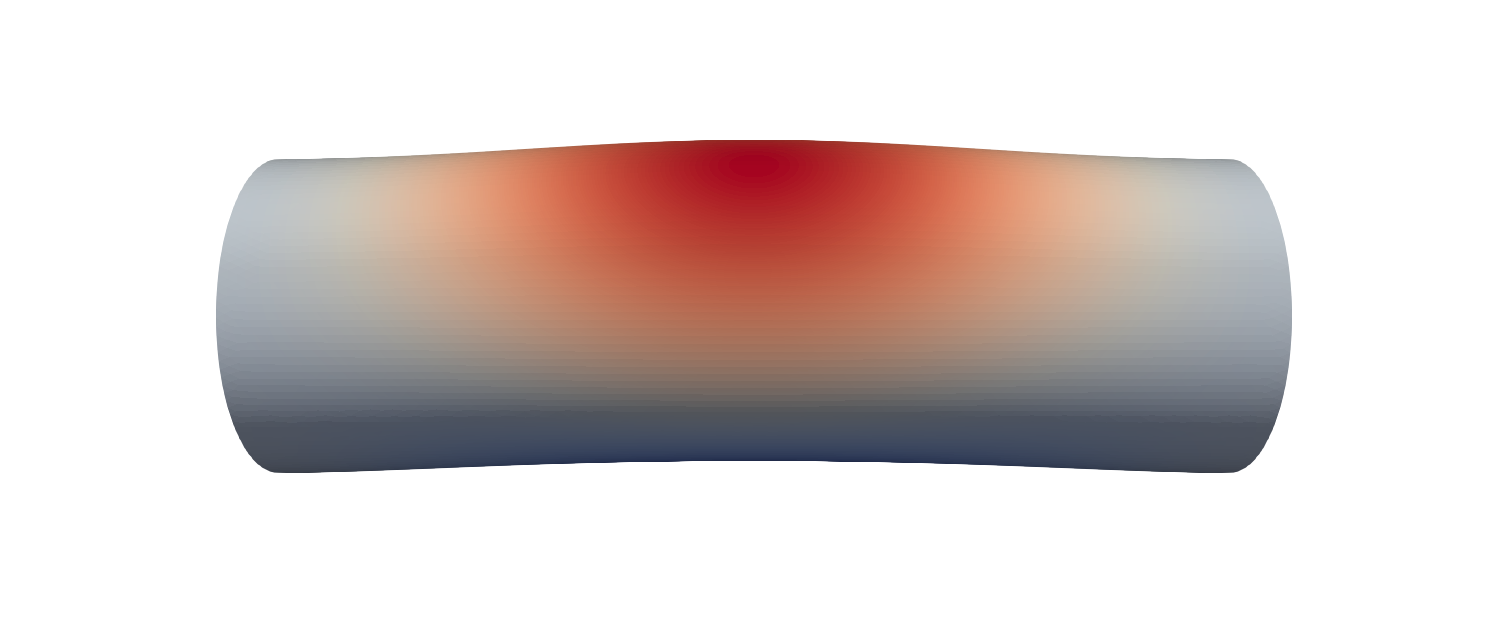}
    \caption{An image showing $\go(\rho u)$, where $u$ solves the line forcing problem outlined in Section \ref{sec:exp:lineForcing}.
    For the purposes of visualisation, we set $\rho = 0.2$.
    Here, the colour corresponds to the value of $u$, with blue being low values and red being high values.}\label{fig:lineForcing}
\end{figure}

\subsection{Phase field}\label{sec:exp:phaseField}
Here we consider that $c_0(\phi) = 20 \phi$ for the unknown phase function $\phi$.
The potential $W$ is given by the obstacle potential,
\begin{equation}
    W(s) =
    \begin{cases}
        1-s^2 & |s|\leq 1,\\
        \infty & |s|>1.
    \end{cases}
\end{equation}

For our numerical experiment, we will consider that $\phi_h \in \Phi_h$, where $\Phi_h$ which is given by
\begin{equation}
    \Phi_h :=
    \left\{
        \varphi_h \in C^0(\goh) : \varphi_h|_T \in \mathbb{P}^1(T) \forall T \in \mathcal{T}_h,\, |\varphi_h| \leq 1,\, \int_\goh c_0(\varphi_h) \dee \goh = 0
    \right\}.
\end{equation}
The (discrete) PDE constrained optimisation problem which we consider is find $\phi_h \in \Phi$ which minimises $\phi_h \mapsto J_{\rm pf}(u_h(\phi_h),\phi_h)$ where $u_h(\phi_h) \in V_h^p$ solves
\begin{equation}
    a_h(u_h,v_h) = \int_\goh \Delta_\goh v_h c_0(\phi_h) \dee \goh \quad \forall v_h \in V_h^p,\, \int_\go v_h \dee \goh = 0, \quad \int_\goh u_h \dee \goh = 0.
\end{equation}
To minimise the map $\phi_h \mapsto J_{\rm pf}(u_h(\phi_h),\phi_h)$, we utilise the aforementioned VMPT \cite{BlaRup17}.
The resulting obstacle problem is solved by use of a primal-dual active set method \cite{BlaGarSar13}.

For the experiment, we choose $L = 6$ and $R = 1$, as well as setting $b = 16$ and $\epsilon = 2^{-3}$.
The minimisation method is started with the initial $\varphi_h^0$ to be the interpolation of $2^{-5} \cos(\frac{k \pi }{L} x_3)$ for a selection of $k$.
While the initial condition is axi-symmetric, the mesh is not, constructed with \texttt{gmsh} \cite{gmsh}.
We perform the VMPT process until $e_k = \varphi_h^{k+1}- \varphi_h^k$ the difference between subsequent iterations satisfy $\epsilon \|\nabla e_k \|_{L^2(\goh)}^2 + \epsilon^2 \|e_k\|_{L^2(\goh)}^2 \leq 10^{-3}$.
In Figure \ref{fig:phaseForce} we see a picture of the numerical example.
\begin{figure}
    \centering
    \includegraphics[width = .45 \linewidth]{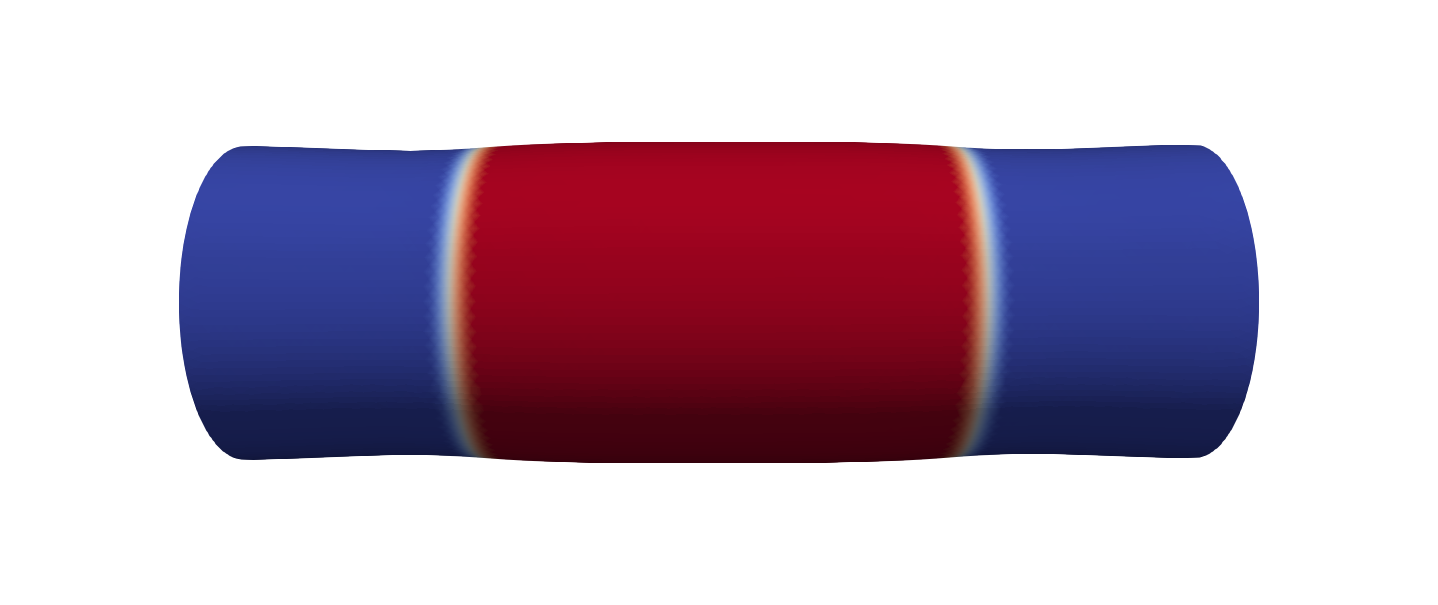} \hfill
    \includegraphics[width = .45 \linewidth]{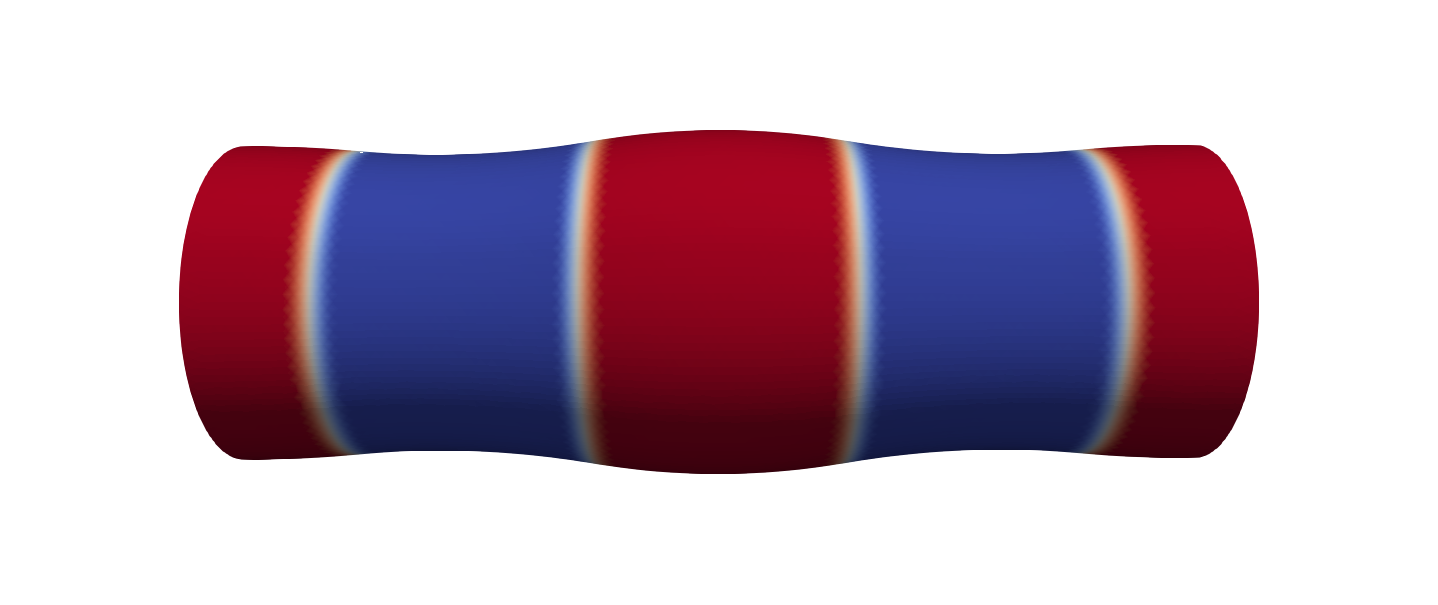} \\
    \includegraphics[width = .45 \linewidth]{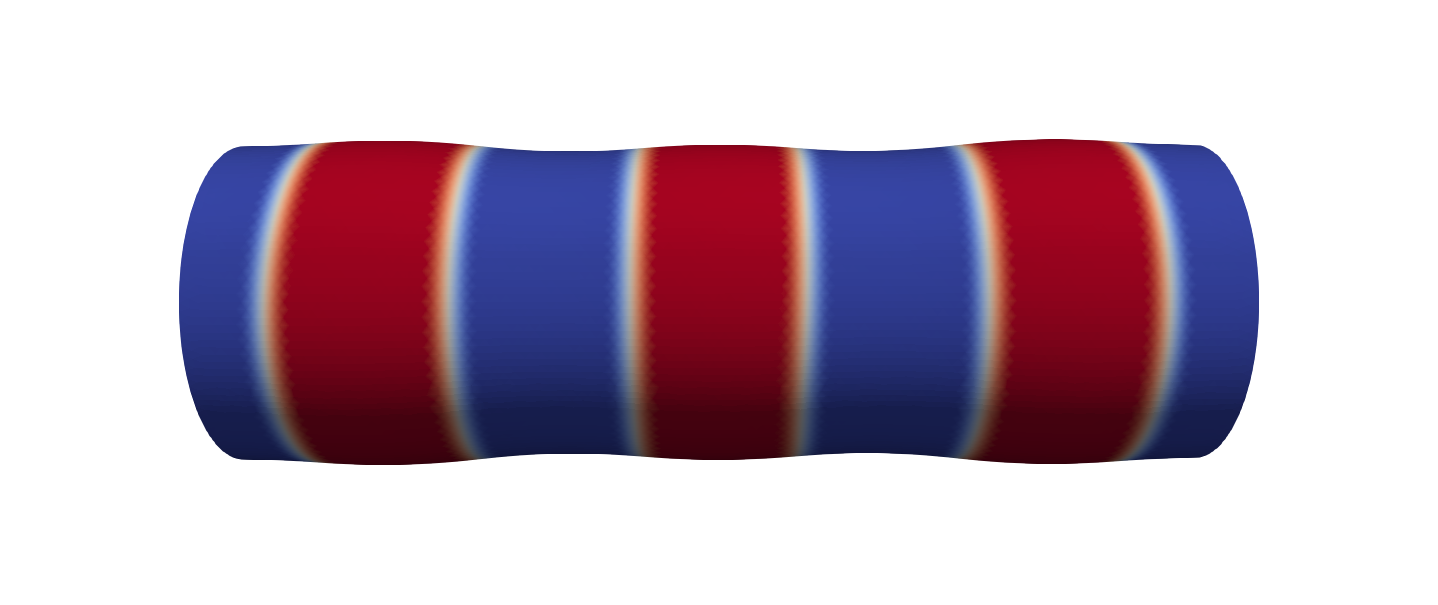} \hfill
    \includegraphics[width = .45 \linewidth]{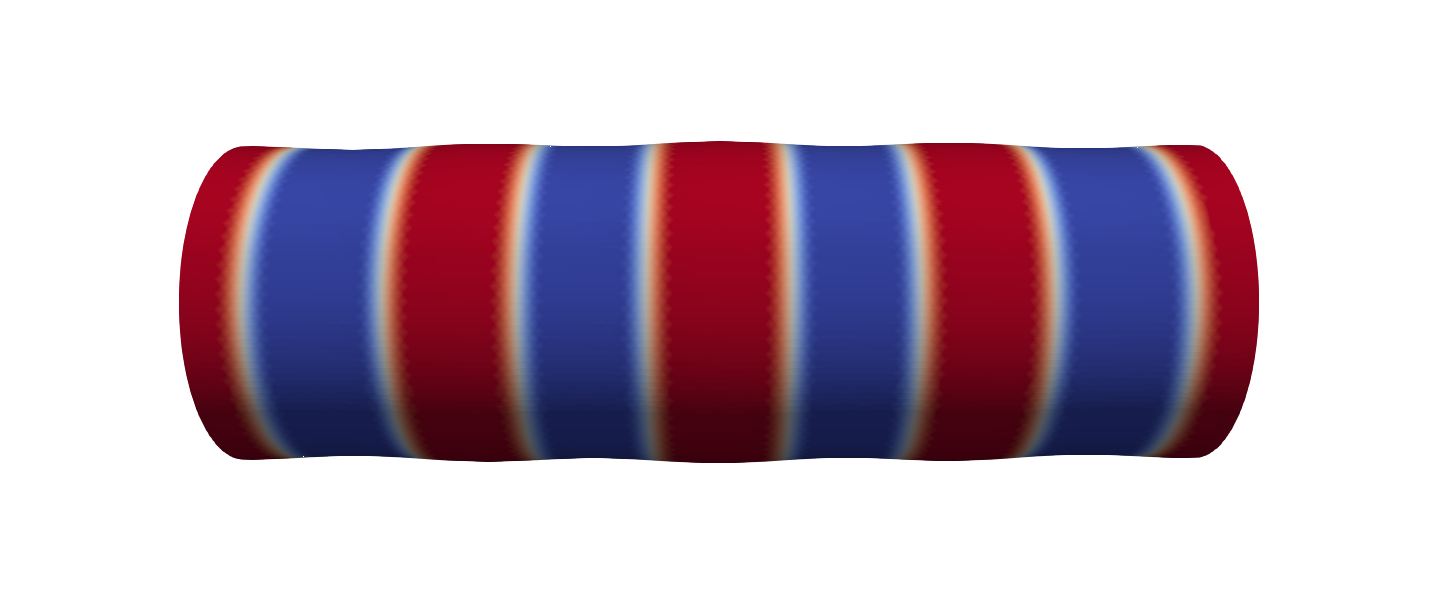}
    \caption{An sequence of images showing $\go(\rho u)$, where $u$ solves the phase field forcing problem outlined in Section \ref{sec:exp:phaseField}.
    From top left to bottom right we show the resulting image with $k = 2, 4, 6, 8$.
    As the number of interfacial regions increases, the energy increases.
    For the purposes of visualisation, we set $\rho = 0.2$.
    Here, unlike the other experiments, the colour corresponds to the value of $\varphi$, with blue corresponding to areas where $\varphi\approx -1$ and red corresponding to areas with $\varphi \approx 1$.}\label{fig:phaseForce}
\end{figure}


\section{Conclusion}\label{sec:Conclusion}
We have constructed a quadratic energy which models a tubular biomembrane which is near a Helfrich-Cylinder and shown well-posedness of associated minimisation problems.
We have considered a number of different forces which would induce a small deformation.
In the appendix we have given a result which shows the standard $H^2$ norm on a surface may be controlled by a more convenient norm which does not contain mixed derivatives.

As an outlook, it may be interesting to apply such tools to surfaces which are near-catenoids, or sections of spheres.
It could be worthwhile to also consider small perturbations of the boundary conditions, e.g., instead of $u|_{\partial\go} = 0$, to have $u|_{\partial \go} = \rho \hat{u}$ for some $\hat{u}$ being sufficiently regular.
Furthermore, it would be of interest to develop numerical estimates for surfaces with boundaries.

\section*{Acknowledgments}
PJH would like to thank the Freie Universit\"at Berlin Research Alumni Program which supported a visit to CG.

\appendix
\section{Equivalence of norms on \texorpdfstring{$H^2(\Gamma)$}{H2Gamma}}\label{sec:appendixA}
We provide the following lemma which is an adaptation of Lemma 3.2 in \cite{DziEll13}.
\begin{lemma}\label{lem:EquivNormDirichlet}
	Let $\Gamma$ be an $n$ dimensional hypersurface with sufficiently smooth boundary.
	For $v \in H^2_0(\Gamma)$, it holds that there is $C>0$ independent of $v$ such that
	\[
		|v|^2_{H^2(\Gamma)} \leq \|\Delta_\Gamma v\|_{L^2(\Gamma)}^2 + C\|\nabla_\Gamma v\|_{L^2(\Gamma)}^2.
	\]
\end{lemma}
\begin{proof}
	We assume without loss of generality that $v$ is sufficiently smooth, say $v \in H_0^1(\Gamma) \cap C^\infty (\bar{\Gamma})$.
	We let $\conormal \colon \partial\Gamma \to \R^3$ the outward unit conormal to $\Gamma$.
	We calculate, summing over repeated indices,
	\begin{align*}
		|v|_{H^2(\Gamma)}
		=&
		\int_\Gamma \D{i}\D{j} v \D{i}\D{j}v
		=
		-\int_\Gamma \D{j}v \D{i}\D{i}\D{j}v + \int_{\partial\Gamma} \D{i}\D{j} v \D{j} v \conormal_i
		\\
		=&
		-\int_\Gamma \D{j}v \D{j}\D{i}\D{i}v + \int_{\partial\Gamma} \D{i}\D{j} v \D{j} v \conormal_i
		+ \int_\Gamma \D{j}v \left( \D{i}\D{i}\D{j}v - \D{j}\D{i}\D{i}v\right)
		\\
		=&
		\int_{\Gamma} \D{i}\D{i} v \D{j}\D{j}v
		+ \int_{\partial\Gamma} \left( \D{i}\D{j} v \D{j} v \conormal_i - \D{i}\D{i} v \D{j}v \conormal_j\right)
		\\
        &+ \int_\Gamma \D{j}v \left( \D{i}\D{i}\D{j}v - \D{j}\D{i}\D{i}v\right).
	\end{align*}
	The first term in the final line of the above is $\|\Delta v \|_{L^2(\Gamma)}^2$ and the second term are boundary terms which vanish since $\D{j}v = 0$ on $\partial \Gamma$.
    For the third term, we calculate using \cite[Lemma 2.6]{DziEll13}
    \begin{align*}
        \D{i}\D{i}\D{j} v
        =&
        \D{i}(\D{i}\D{j} v - \D{j}\D{i}v)
        +
        (\D{i}\D{j} - \D{j}\D{i})\D{i} v
        +
        \D{j}\D{i}\D{i} v
        \\
        =&
        \D{j}\D{i}\D{i} v
        +
        \D{i}\left( \left(\mathcal{H}\nabla_\Gamma v \right)_j \nu_i - \left(\mathcal{H}\nabla_\Gamma v \right)_i \nu_j \right)
        \\
        &+
        \left( \left(\mathcal{H}\nabla_\Gamma \D{i}v \right)_j \nu_i - \left(\mathcal{H}\nabla_\Gamma \D{i}v \right)_i \nu_j \right)
        \\
        =&
        \D{j}\D{i}\D{i} v
        +
        \D{i}\left( \left(\mathcal{H}\nabla_\Gamma v \right)_j  \right) \nu_i
        +
        \left(\mathcal{H}\nabla_\Gamma v \right)_j \D{i} \nu_i
        -
        \D{i}\nu_j \left(\mathcal{H}\nabla_\Gamma v \right)_i
        \\
        &-
        \D{i}\left(\mathcal{H}\nabla_\Gamma v \right)_i \nu_j
        +
        \left(\mathcal{H}\nabla_\Gamma \D{i}v \right)_j \nu_i
        - \left(\mathcal{H}\nabla_\Gamma \D{i}v \right)_i \nu_j,
    \end{align*}
    where we further swap derivatives to obtain
    \begin{align*}
        \left(\mathcal{H}\nabla_\Gamma \D{i}v \right)_j \nu_i
        =&
        \nu_i \mathcal{H}_{jk} \D{k}\D{i}v
        \\
        =&
        \nu_i \mathcal{H}_{jk} \left( \D{i}\D{k} v + \left(\mathcal{H} \nabla_\Gamma v \right)_i \nu_k - \left( \mathcal{H}\nabla_\Gamma v \right)_k \nu_i \right)
        =
        - (\mathcal{H}^2\nabla v)_j.
    \end{align*}
    Combining this, we have that
    \begin{equation*}
        \D{j} v \D{i}\D{i}\D{j} v
        =
        \D{j} v \left( \D{j}\D{i}\D{i} v
        +
        H \left(\mathcal{H}\nabla_\Gamma v \right)_j
        -
        2 \left(\mathcal{H}^2 \nabla_\Gamma v \right)_i
        \right).
    \end{equation*}
	This has now shown that
	\begin{align*}
		|v|_{H^2(\Gamma)}^2
		\leq&
		\|\Delta_\Gamma v\|_{L^2(\Gamma)}^2
		+
		\| \mathcal{H}H -2 \mathcal{H}^2\|_{L^\infty(\Gamma)} \|\nabla_\Gamma v\|_{L^2(\Gamma)}^2.
	\end{align*}
\end{proof}

\section{Variations of geometric quantities of interest}\label{sec:appendix}
In this appendix, we provide the variations of surface functionals.
The results are simplified from \cite{EllFriHob17}.

We recall that $W(\Gamma):= \frac{1}{2}\int_\Gamma H^2 \!\!\dee \Gamma$ and $A(\Gamma):= \int_\Gamma 1 \!\!\dee \Gamma$ and we require their first and second variations.
Throughout, we assume that $\Gamma$ is a smooth hypersurface with boundary and that $u \in C^2$.
\begin{proposition}\label{prop:firstVar}
    The first variations of $W$ and $A$ in direction $u\nu$ are given by
    \begin{equation*}\begin{split}
        W'(\Gamma)[u\nu]
        =&\
        \int_\Gamma \left( -H \Delta_\Gamma u +  \frac{1}{2}H^3 u - H|\mathcal{H}|^2 u \right) \!\! \dee \Gamma
        \\
        A'(\Gamma)[u\nu]
        =&\
        \int_\Gamma H u \!\! \dee \Gamma.
    \end{split}\end{equation*}
    Considering $\Gamma = \go$, it holds that
    \begin{equation*}\begin{split}
        W'(\go)[u\nu]
        =&\
        \frac{1}{R}\int_\go \left( -\Delta_\go u +  \frac{1}{2 R^2}u \right) \!\! \dee \go
        \\
        A'(\go)[u\nu]
        =&\
        \frac{1}{R} \int_\go u \!\! \dee \go.
    \end{split}\end{equation*}
\end{proposition}

The second variations are more challenging as highlighted in \cite[Remark 3.2]{EllFriHob17}.
For geometric variations of generic functionals on generic surfaces it need not hold that taking the second variation coincides with taking the first variation twice.
However, when when considering functionals which are independent of parametrisation on closed surfaces, the two strategies coincide.
In this work, we consider only the second variations of functionals which do not depend on parametrisation, $W$ and $A$, however we consider that $\Gamma$ has a boundary.
The results within the appendix when combined with Remark 3.2 of \cite{EllFriHob17} states that the discrepancy between the two strategies depends only on the variation in the tangential direction $-u \nabla_\Gamma u$.
\begin{proposition}
    The first variations of $W$ and $A$ in direction $u \nabla_\Gamma u$ are given by
    \begin{equation*}\begin{split}
        W'(\Gamma)[u\nabla_\Gamma u]
        =&\
        \frac{1}{2}\int_\Gamma \left( 2H v \nabla_\Gamma H \cdot \nabla_\Gamma v + H^2(v\Delta_\Gamma v + |\nabla_\Gamma v |^2 ) \right)\!\!\dee \Gamma
        \\
        A'(\Gamma)[u\nabla_\Gamma u]
        =&\
        \int_\Gamma \left( v\Delta_\Gamma v + |\nabla_\Gamma v|^2 \right)\!\!\dee \Gamma.
    \end{split}\end{equation*}
    For $u$ with $u|_{\partial \Gamma} = 0$, it holds that
    $W'(\Gamma)[u\nabla_\Gamma u] = A'(\Gamma)[u\nabla_\Gamma u] = 0$.
\end{proposition}
From which, we may calculate the second variations as follows.
\begin{proposition}\label{prop:secondVar}
    The second variations of $W$ and $A$ in direction $(u\nu, u\nu)$ are given by
    \begin{equation*}\begin{split}
        W''(\Gamma)[u\nu,u\nu]
        =&\
        \int_\Gamma \bigg( (\Delta_\Gamma u + |\mathcal{H}|^2 u)^2 + 4 H \mathcal{H}:u D^2_\Gamma u + 2 H\mathcal{H}\nabla_\Gamma u \cdot \nabla_\Gamma u - 3 H^3 u \Delta_\Gamma u
		\\
		&\quad\quad\quad
        - \frac{3}{2}H^2 |\nabla_\Gamma u|^2 +
            \left(2H \Tr(\mathcal{H}^3)- \frac{5}{2}H^2 |\mathcal{H}|^2 + \frac{1}{2}H^4 \right)u^2
        \bigg) \!\!\dee \Gamma
        \\
        A''(\Gamma)[u\nu,u\nu]
        =&\
        \int_\Gamma \left( |\nabla_\Gamma u|^2 + \left(H^2 - |\mathcal{H}|^2\right)u^2 \right)\!\!\dee \Gamma.
    \end{split}\end{equation*}
    Considering $\Gamma = \go$, it holds that
    \begin{equation*}\begin{split}
        W''(\go)[u\nu,u\nu]
        =&\
        \int_\go \bigg( (\Delta_\go u)^2 - \frac{2}{R^2}|\tau \cdot \nabla_\go u|^2 - \frac{1}{2 R^2} |\nabla_\go u|^2 + \frac{1}{R^4} u^2
        \bigg) \!\!\dee\go,
        \\
        A''(\go)[u\nu,u\nu]
        =&\
        \int_\go |\nabla_\go u|^2  \!\!\dee\go.
    \end{split}\end{equation*}
\end{proposition}

\bibliographystyle{alphaTwo}
\bibliography{bib}

\end{document}